\documentclass[11pt,regno]{amsart}

\usepackage{amsmath}
\usepackage{amssymb,mathtools}
\usepackage[english]{babel}
\usepackage{enumerate}
\usepackage{tikz}
\usepackage[all]{xy}
\usepackage{geometry}
\usepackage{ulem}
\usepackage{multicol}
\usepackage{hyperref}
\hypersetup{
    colorlinks,%
    citecolor=blue,%
    filecolor=blue,%
    linkcolor=blue,%
    urlcolor=blue}

\usepackage[utf8]{inputenc}
\usepackage{tikz-cd}
\usepackage[T1]{fontenc}
\usepackage{amssymb}
\usepackage{graphicx}
\usepackage{colonequals}
\usepackage{multicol}
\usepackage{amsthm}
\usepackage{hyperref}
\newtheorem{thm}{Theorem}[section]
\newtheorem{lem}[thm]{Lemma}
\newtheorem{df}[thm]{Definition}

\newtheorem{prop}[thm]{Proposition}
\newtheorem{cor}[thm]{Corollary}

\newtheorem{rmk}[thm]{Remark}
\newtheorem{nt}[thm]{Notation}
\usepackage{enumitem}

\newtheorem*{thm*}{Theorem}

\newcommand{\OO}{\mathcal{O}}
\theoremstyle{definition}
\usepackage{amsaddr}
\usepackage{comment}

\newcommand{\CC}{\mathbb{C}}

\DeclareMathOperator{\Aut}{Aut}
\DeclareMathOperator{\Autgr}{Aut_{gr}}

\DeclareMathOperator{\Bir}{Bir}

\DeclareMathOperator{\Pic}{Pic}

\DeclareMathOperator{\im}{Im}
\DeclareMathOperator{\Sing}{Sing}
\DeclareMathOperator{\Proj}{Proj}

\definecolor{myorchid}{HTML}{a4538a}

\usepackage[english]{babel}

\usepackage{geometry}
\title{Automorphisms of hypersurfaces in polarized varieties}

\author{Ana Quedo}
\email{ana.quedo@gmail.com}

\begin{document}

\maketitle

\begin{abstract}
Let $F$ be a projective variety, let $H$ be a normally generated line bundle on $F$, and let $X\in |H^{\otimes d}|$ be a general projective variety. We establish a uniform criterion, formulated in terms of the graded Betti numbers of the section ring $R(F,H)$, that provides sufficient conditions for every automorphism of $X$ to extend to an automorphism of $F$. This extends the classical results of Matsumura--Monsky on automorphisms of hypersurfaces in projective spaces to a broader class of polarized varieties $(F,H)$, including rational homogeneous varieties of Picard number one, smooth Fano threefolds with their anticanonical polarization, and examples of hyperkähler and Calabi--Yau varieties.

\end{abstract}

\section{Introduction}

\subsection{The classical Matsumura--Monsky theorem}

The classical Matsumura--Monsky \cite{MatsumuraMonsky} theorem states that for a smooth hypersurface $X \subseteq \mathbb{P}^{n}$ with degree $d \geq 3$ and $n \geq 3$, $\Aut(X)$ can be identified with the stabilizer of $X$ in $\mathrm{PGL}(n+1)$,  with the sole exception of smooth quartic surfaces in $\mathbb{P}^3$. In particular, the study of automorphisms of $X$ reduces to the study of linear transformations of the ambient projective space preserving the defining equation of $X$.

They also proved that smooth quartic surfaces in $\mathbb{P}^3$ are the only hypersurfaces with degree $d \geq 3$ and $n \geq 3$ whose automorphism group may not be finite and that for a very general smooth hypersurface $X$, the automorphism group  $\operatorname{Aut}(X)$ is trivial.

Among its important applications, the Matsumura--Monsky theorem is one of the key ingredients in the construction, via Geometric Invariant Theory, of moduli spaces of smooth hypersurfaces of fixed degree; see, for instance, \cite[Chapter 5]{Mukai03}. It is also worth noting that in many cases one has $\Aut(X) = \Bir(X)$, hence the above result also provides  information about the birational geometry of hypersurfaces.

The classical extension result has also been generalized to complete intersections. Benoist \cite{Benoist} proved that any automorphism of a smooth complete intersection of dimension at least three in $\mathbb{P}^{n}$ extends to an automorphism of the ambient projective space. Similarly, the generic triviality of the automorphism group has recently been established for complete intersections under suitable dimension and degree bounds \cite{Automorphismandcohomology}.

The importance of the Matsumura--Monsky theorem lies not only in its consequences, but also in the questions it raises. For instance, one may ask which finite groups can occur as automorphism groups of smooth hypersurfaces. This question was studied for Klein hypersurfaces in \cite{Klein}. One may also ask for effective bounds on the order of the automorphism group; see, for example, \cite{EsserLi} and \cite{YangYuZhu2025}.

Another problem is to ask which automorphisms of smooth quartic surfaces are induced by birational transformations of $\mathbb{P}^3$. This question, known as Gizatullin's problem, has attracted considerable attention in recent years (see, for example, \cite{OguisoCremonaFields,Oguiso2,APZ,paivaquedo}). This problem can also be extended to the case where $F$ is a Fano threefold of Picard number 1 and $S$ is a smooth element of $|-K_F|$, as done in
\cite{AGP}.

It is also worth to understand to what extent similar phenomena hold for hypersurfaces in more general ambient varieties. Esser recently proved that every automorphism of a quasismooth hypersurface $X$ in a well-formed weight projective space $\mathbb{P}(a_1,\dots, a_n)$ can be extended to an automorphism of the ambient space if $\dim X \geq 3$ or $\dim X = 2$ and $\deg X \neq a_1 + \dots + a_n$, see \cite{EsserAutoWPS}.
Benedetti and Manivel proved an analogous result for hyperplane sections of generalized Grassmannians with Picard number one \cite{BenedettiManivel}.  These results suggest that the extension phenomenon is not intrinsically tied to projective space and motivate the following broader question:

\begin{quote}
To what extent does an automorphism of a smooth hypersurface in a more general ambient variety extend to an automorphism of the ambient variety?
\end{quote}

\subsection{Extending automorphisms of hypersurfaces in more general ambient spaces}

Let $F$ be a smooth projective variety and let $X\subset F$ be a smooth hypersurface. Denote by $\operatorname{Aut}(F;X)$ the subgroup of automorphisms of $F$ preserving $X$. We are interested in understanding when the natural restriction homomorphism
$\Aut(F;X) \to \Aut(X)$
is an isomorphism. When this happens, information about the automorphism group of $X$ can be obtained from the automorphism group of the ambient variety $F$, which is often better understood.

There are several settings in which this perspective is particularly promising. First, suppose that $F$ is an irreducible holomorphic symplectic variety (also called a hyperkähler variety). For such varieties, the group $H^2(F,\mathbb{Z})$ carries a natural lattice structure, and the automorphism group of $F$ can be described in terms of isometries of this lattice satisfying suitable additional conditions. If every automorphism of $X$ extends to $F$, then $\operatorname{Aut}(X)$ can likewise be studied through isometries of the lattice $H^2(F,\mathbb{Z})$. Moreover, if $X\subset F$ is a smooth hypersurface of dimension at least $3$, the Lefschetz hyperplane theorem gives $H^2(X,\mathbb{Z})\simeq H^2(F,\mathbb{Z}).$
Thus, in this setting, the automorphism group of $X$ can, in favorable cases, be studied in terms of the lattice $H^2(X,\mathbb{Z})$ itself.

 Another interesting class of ambient varieties is that of Fano varieties. Smooth Fano threefolds with infinite automorphism groups have been completely classified by Przyjalkovski, Chelstov and Shramov \cite{Fano3InftyAuto}. An extension theorem for automorphisms could therefore provide significant insight into $\operatorname{Aut}(X)$ when $X$ is a hypersurface in a Fano threefold. This may also aid in understanding the automorphism groups of Fano fourfolds, as many Fano threefolds naturally appear as hypersurfaces within them.
 
Finally, one may expect that many of the applications and questions raised by the classical Matsumura--Monsky theorem for hypersurfaces in $\mathbb{P}^n$ admit analogous formulations and answers for hypersurfaces in more general ambient varieties.

\subsection{Our results}
Our theorem applies to several classes of polarized varieties. We briefly describe some of the main examples.

\begin{enumerate}
\item \textbf{Fano threefolds (Theorem \ref{Fano}).}
Let $F$ be a normal factorial Fano threefold with terminal singularities, with $-\omega_F$ very ample, and that $F$ is not a quartic hypersurface in $\mathbb{P}^4$. Let $X$ be a very general projective variety $X \in |\omega_F^{-d}|$ with $d\geq 4$. Then, \[\operatorname{Aut}(X)
    \simeq
    \operatorname{Aut}(F;X)
    \simeq
    \operatorname{Aut}(\mathbb P^N;X);
    \]
where $i \colon F \hookrightarrow \mathbb{P}^N$ is the anticanonical embedding.

\item \textbf{Rational homogeneous varieties of Picard number one (Theorem \ref{rational}).}
Let $F$ be a rational homogeneous variety of Picard number one and dimension at least $3$, and let $H$ be the ample generator of $\operatorname{Pic}(F)$. This includes, for example, Grassmannians, Lagrangian Grassmannians $\mathrm{LG}(n,2n)$, and spinor varieties.  Then, let $X$ be a smooth projetive variety,  $X \in |H^{\otimes d}|$ for $d >2$, $X \notin |\omega_F^{-1}|$. Then,
$$\operatorname{Aut}(X)
    \simeq
    \operatorname{Aut}(F;X)
    \simeq
    \operatorname{Aut}(\mathbb P^N;X);$$
    
where $i \colon F \hookrightarrow \mathbb{P}^N$ is the embedding induced by $H$. This complements the result of Benedetti--Manivel \cite{BenedettiManivel}, which treats the case $d=1$ of hyperplane sections.

\item \textbf{Hyperkähler and Calabi--Yau varieties (Theorem \ref{trivialcanonical}).} Let $F$ be a smooth simply connected projective variety with trivial canonical bundle and $n:=\dim F \geq 3$ and $H$ a normally generated line bundle on $F$. For $d>n+2$, let $X \in |H^{\otimes d}|$ a smooth projective variety. Then
$$\Aut(X)\simeq \Aut(F;X) \simeq\Aut(\mathbb{P}^N;X),$$ 
 where $i \colon F \hookrightarrow \mathbb{P}^N$ is the embedding induced by the complete linear system $|H|$.  
For certain specific examples of hyperkähler and Calabi-Yau varieties this bond of $d>n+2$ can be improved.
For certain families of hyperkähler varieties of type $K3^{[2]}$, we obtain a sharper bound as a consequence of a recent work of Ortiz, Rojas, and Song \cite{ORS} (see Corollary \ref{hyperkahler}).
In the Calabi--Yau case, we carry out explicit computations for Wehler varieties (See Corollary \ref{Wehler2}).

\end{enumerate}

These applications show that the extension phenomenon discovered by Matsumura and Monsky for hypersurfaces in projective space is not an isolated feature of $\mathbb{P}^n$. Rather, it can be viewed as part of a broader principle: under suitable positivity, generation, and homological conditions on a polarized variety $(F,H)$, sufficiently general hypersurfaces in $F$ inherit their automorphisms from the ambient variety.

%To prove our results, we draw inspiration from the original proof of Matsumura--Monsky and from Esser's recent work. The original argument of Matsumura--Monsky relies on the fact that, for a hypersurface $X\subseteq \mathbb{P}^N$ defined by a homogeneous polynomial $f$ of degree $d\geq 3$, with $d\neq 4$, the polarization $\mathcal{O}_X(1)$ is preserved by every automorphism of $X$. Consequently, an automorphism of $X$ induces an automorphism of the section ring
%$$R(X, \OO_X(1)) \simeq \mathbb{C}[T_0, \dots, T_N]/(f)$$

%\noindent and, under the hypotheses of the Matsumura--Monsky theorem, such an automorphism lifts to an automorphism of the polynomial ring $\mathbb{C}[T_0,\ldots,T_N]$. This in turn induces an automorphism of the ambient projective space preserving $X$.

%Esser's strategy adapts this approach to weighted projective spaces, replacing the standard polynomial ring by a graded polynomial ring whose variables have different weights. Additional arguments are required in that setting to deal with the singularities of the ambient weighted projective space and, in particular, with the distinction between Cartier and Weil divisors.

Our approach is inspired by the original proof of Matsumura--Monsky and by Esser's recent work. The main difference is that we work with the section ring of a polarized variety, whereas their arguments are formulated in terms of the polynomial ring $\mathbb{C}[T_0, \dots, T_N]$ or its graded version.

Let $F$ be a projective variety and let $H$ be an ample line bundle on $F$. We consider the section ring $$R(F, H) = \bigoplus_{m\geq 0} H^0(F, H^{\otimes m})
.$$

%For a hypersurface $X\in |H^{\otimes d}|$, defined by a section $f\in H^0(F,H^{\otimes d})$, the section ring of $X$ is naturally related to the quotient of $R(F,H)$ by the element $f$. Under the assumptions of our main theorem, we obtain

\noindent In our case $R(F,H)$ is not a polynomial ring, but it is still a finitely generated $\mathbb{C}$-algebra, and as a consequence $R(F,H)$ can be written as $$R(F,H) \simeq \frac{\mathbb{C}[T_0, \dots, T_N]}{I}.$$

When $H$ is a normally generated line bundle, $\deg T_i=1$ and $I$ is a homogeneous ideal. Moreover, the degrees of a minimal homogeneous generating set of $I$ are homological invariants that can be computed through the graded Betti numbers of $R(F,H)$. In our problem, we will be particularly interested in the maximal degree of the generators of a minimal homogeneous generating set of $I$, that will be denoted by $t_1(R(F,H))$. Our main theorem is the following.

\begin{thm}[Theorem \ref{thm:main}]
 \label{thm:main}
 Let $F$ be a projective variety and $H$ normally generated line bundle on $F$ and $R(F,H)$ its section ring. Let $X \in |H^{\otimes d}|$ a projective variety, for $d >t_1(R(F,H))$. Suppose 
 \begin{enumerate}
     \item $H^1(F,H^{\otimes(m-d)})=0$ for all $m\geq0$,
     \item Every $g \in \Aut(X)$ preserves $H|_X$,
     \item $\Aut(F;X)\subseteq \Aut(F,H).$

 \end{enumerate}
 Then, $$\Aut(X)\simeq \Aut(F;X) \simeq\Aut(\mathbb{P}^N;X)$$
where $i_H \colon F \hookrightarrow \mathbb{P}^N$ is the embedding induced by the complete linear system $|H|$.

\end{thm}

The proof of the theorem can be summarized in four steps.

\medskip

\noindent\textbf{Step 1.}
Let $X=V(f)\subset F$ be a hypersurface with $f\in H^0(F,H^{\otimes d})$. Under the hypotheses of the theorem, we prove that $$R(X, H|_X) \simeq R(F,H)/(f).$$

\medskip

\noindent\textbf{Step 2.}
Assume that $d> t_1(R(F,H))$.
We prove that every graded automorphism of $R(F,H)/(f)$ lifts to a graded automorphism of $R(F,H)$ preserving the ideal $(f)$. More precisely, if $$\Autgr(R(F,H), R(F,H)/(f))$$ denotes the group of graded automorphisms of $R(F,H)$ preserving $(f)$, then the natural map
$$ \Autgr(R(F,H), R(F,H)/(f)) \simeq \Autgr( R(F,H)/(f))$$
is an isomorphism.

\medskip

\noindent\textbf{Step 3.}
When $H$ is normally generated, the embedding of $F$ defined by $H$ is projectively normal. Consequently, $$ \Aut(F,H) \simeq \frac{\Autgr(R(F,H))}{\mathbb{C}^*},$$

\noindent where $\operatorname{Aut}(F,H)$ denotes the subgroup of automorphisms of $F$ preserving $H$.

\medskip

\noindent\textbf{Step 4.}
Finally, under the hypotheses of the theorem, every automorphism of $X$ preserves the polarization $H|_X$. Therefore, putting all the above steps together, we obtain: 
    $$ \Aut(X) \simeq \frac{\Autgr( R(F,H)/(f))}{\mathbb{C}^*} \simeq \frac{\Autgr(R(F,H), R(F,H)/(f))}{\mathbb{C}^*} \simeq \Aut(F,X).$$

The paper is organized as follows. In Section \ref{2}, we recall the notions of normally generated and normally presented line bundles, together with the commutative algebra results that will be used throughout the paper. In Section \ref{3}, we prove the lifting property for graded algebras that underlies Step 2. In Section \ref{4}, we establish the correspondence between graded automorphisms of the section ring and automorphisms of the polarized variety. Section \ref{5} is devoted to understanding under which hypotheses every automorphism $g\in\operatorname{Aut}(X)$ preserves the polarization $H|_X$. In Section \ref{sec.Extaut} we state and prove our main results, and in Section \ref{7}, we discuss their main applications.

\medskip

\textbf{Acknowledgements.} This project is one of the research outcomes of PRIN 2022, and I would like to express my sincere gratitude to the rest of the PRIN team for their support and encouragement throughout this process.

First and foremost, I would like to thank Antonella Grassi for suggesting this project to me and for bringing to my attention several vanishing theorems that helped me develop it. I am also grateful to Andrea Petracci for pointing out mistakes in my initial ideas and for his valuable comments. I would like to thank Concettina Galati for carefully reading the first drafts and for her insightful questions. I am especially grateful to all three of them for taking the time to listen to my ideas and for engaging in many fruitful discussions throughout the development of this project.

I would also like to express my gratitude to Carolina Araujo and Alex Massarenti for their comments, as well as their support on the final stages of this project.

This project was supported by funding from the European Union – NextGenerationEU under the National Recovery and Resilience Plan (PNRR), Mission 4 “Education and Research”, Component 2 “From Research to Business”, Investment 1.1, PRIN 2022 “Geometry of Algebraic Structures: Moduli, Invariants, Deformations”, DD No. 104, 2 February 2022, proposal code 2022BTA242, CUP J53D23003720006. During the final stages of the project, I was supported by CNPq (Conselho Nacional de Desenvolvimento Científico e Tecnológico), grant No. 200313/2026-6.

\section{Preliminaries}
\label{2}

Throughout this paper, the term projective variety refers to an integral projective scheme over $\mathbb{C}$.

\subsection{Normal generation and Normal presentation ($N_0$ and $N_1$ property)} 

We start by recalling the concepts of normally generated and normally presented line bundles, these concepts were introduced by Mumford in \cite{Mumford}, and we refer the reader to this paper for a more detailed exposition.

Consider $F$ a complex projective variety and
$B$ a very ample line bundle of $F$. In this case, we know that there is a morphism $\phi_B$ that associates each point $x \in F$ to a point in the projective space of hyperplanes in  $H^0(F, B)$, denoted here by $\mathbb{P}(H^0(F, B) )$.
\begin{align*} 
\phi_B \colon F &\to \mathbb{P}(H^0(F, B) )\\
x &\mapsto \phi_B(x):= \{s \in H^0(F, B)  \text{ | }  s(x)=0 \}.
\end{align*}

Considering a basis $s_0, \dots, s_n$ of  $H^0(F, B)$, we can write

\begin{align*} 
\phi_B \colon F &\to \mathbb{P}(H^0(F, B) )\\
x &\mapsto (s_0(x) \colon \dots : s_n(x))
\end{align*}

 Given a finite dimensional vector space $V$, we denote by $S^d(V)$ the $d^{th}$ symmetric power.

In our case, since $H^0(\mathbb{P}^n, \OO(1)) \simeq H^0(F,B)$, we have $$H^0( \mathbb{P}^n, \OO(d)) \simeq \CC[T_0, \dots, T_n]_d \simeq S^d(H^0(\mathbb{P}^n, \OO(1)) \simeq S^d(H^0(F,B))$$

Since $B \simeq \phi_B^* \OO(1)$, we have the following exact sequence

  \begin{equation} \label{FinPn}
      H^0(\mathbb{P}^n, \OO(d)) \rightarrow H^0(F, B^{\otimes d})\rightarrow H^1(\mathbb{P}^n, I_F(d)) 
  \end{equation}

Therefore, the map $$S^d(H^0(F,B)) \to H^0(F, B^{\otimes d})$$ is well-defined for every $d \geq 0$.

When this map is surjective for $d \geq 0$,
we say $F$ is \textbf{projectively normal} with respect to the embedding
$\phi_B \colon F \hookrightarrow \mathbb{P}^n$ and that the line bundle $B$
is \textbf{normally generated}.
In this case, the homogeneous coordinate $S(F)$ ring associated to the embedding $\phi_B$ is isomorphic to ring of sections of $B$, denoted here by $R(F,B)$. That is:

\[
S(F)
\;:=\;
\frac{\mathbb{C}[T_0,\dots,T_n]}
{\displaystyle \bigoplus_{d \geq 0} H^0\!\left(\mathbb{P}^n,\, I_F(d)\right)} \simeq {\displaystyle \bigoplus_{d \geq 0} H^0\!(F,\, B^{\otimes d})} =: R(F,B) .
\]

When $F$ is a normal variety, $R(F,B)$ is the integral closure of $S(F)$, see \cite[Chapter II, ex 5.14]{Hartshorne}. 

\begin{rmk}
\label{veryampleN0}    
By \ref{FinPn}, we notice that saying that $F$ is projectively normal with respect to $\phi_B$  is equivalent to requiring $H^1( \mathbb{P}^n, I_F(d))=0$ for all $d \geq 0$.
\end{rmk}

\begin{rmk}
    \label{rmk1}
We also notice that $R(F,B)$ being normally generated implies that the ring $R(F,B)$ is generated as a graded algebra by elements of degree 1.
\end{rmk}

We notice that the map $S^d(H^0(X,B)) \to H^0(X, B^{\otimes d})$ is still well-defined when $B$ is an ample line bundle, in fact, more generally we give the following definition.

\begin{df} Let $B$ be an ample line bundle on a complex projective variety $F$. We say $B$ is \textbf{normally generated} if 
$$S^d(H^0(X,B)) \to H^0(X, B^{\otimes d})  \text{ is surjective for every } d \geq 0.$$
\end{df}

\begin{rmk} \label{ngisvample}
If $B$ is normally generated, then $B$ is very ample. For a proof, see
\cite[pages 38,39]{Mumford}
\end{rmk}

In the next definition, we look to the surjectivity of other maps.

\begin{df}
  Let $B$ be a normally generated line bundle on a complex projective variety $F$.
   We say that $B$ is a \textbf{normally presented} line bundle if the following map is surjective for all $d\geq 2$  
  $$H^0(\mathbb{P}^n,\OO(d-2)) \otimes H^0(\mathbb{P}^n, I_F(2)) \to H^0(\mathbb{P}^{n}, I_F(d)).$$

\end{df}

\begin{rmk}
Notice that the above condition  implies that the ideal 
$$I_{*}:=\displaystyle \bigoplus_{d \geq 0} H^0\!\left(\mathbb{P}^n,\, I_F(d)\right)$$
is generated by homogeneous elements of degree 2 (quadrics). 
\end{rmk}

\begin{prop} \label{restrictionNG}
Let $B$ be a normally generated line bundle on a projective variety $F$. \begin{enumerate}
    \item Let $X \subset F$ be a subvariety such that the restriction maps $$H^0(F,B^{\otimes m}) \longrightarrow H^0(X,B|_X^{\otimes m})$$
    are surjective for every $m \geq 0$. Then $B|_X$ is normally generated.

    \item In particular, suppose that $X \in |B^{\otimes d}|$ is a divisor and
    \[
    H^1(F,B^{\otimes m-d})=0
    \qquad\text{for every }m\geq 0.
    \]
    Then $B|_X$ is normally generated.
\end{enumerate}
\end{prop}
\begin{proof}
For the first part, we have that since
the top and vertical maps of the diagram below are surjective, hence the bottom map is surjective.
\[
\begin{tikzcd}
S^m H^0(F,B)
    \arrow[r]
    \arrow[d, two heads]
&
H^0(F,B^{\otimes m})
    \arrow[d, two heads]
\\
S^m H^0(X,B|_X)
    \arrow[r]
&
H^0(X,B|_X^{\otimes m})
\end{tikzcd}
\]

For the second statement, the exact sequence
\[
0\longrightarrow B^{\otimes(m-d)}
\longrightarrow B^ {\otimes m}
\longrightarrow B|_X^ {\otimes m}
\longrightarrow 0
\]
and the vanishing
\[
H^1(X,B^{\otimes (m-d)})=0
\]
implies that
\[
H^0(F,B^{\otimes m})\longrightarrow H^0(X,B|_X ^{\otimes m})
\]
is surjective for every $m\geq 1$. The first statement therefore applies.
\end{proof}

\subsection{Graded Betti numbers and minimal graded free resolutions}
In this subsection, we introduce some notions of commutative algebra needed for formulating our results. We based our exposition in
\cite[Sections 1.10, 20.1, 20.5]{eisenbud1995commutative} and \cite[Section 1B]{Eisenbud2005}.

Let $\mathbb{C}[T_0, \dots, T_N]=:R$ and $R/I$ be a finitely generated algebra. The Hilbert Syzygies Theorem (see \cite[Theorem 1.13, Theorem 20.2]{eisenbud1995commutative}) guarantees the existence of a minimal unique graded free resolution (up to isomorphism) at most $N$ steps.

$$0 \to F_p \to \cdots \to F_2 \to F_1 \to R \to R/I \to 0,
\quad \text{where } F_i = \bigoplus_j R(-j )^{\beta_{i,j}(R/I)} $$

\begin{df}
The integers $\beta_{i,j}(R/I)$ are called \textbf{graded Betti numbers}.
\end{df}

\begin{rmk}
The graded Betti numbers are homological invariants of $R/I$ and do not depend on the choice of the minimal free resolution. In particular,
$\beta_{1,j}(R/I)$ records the number of minimal generators of $I$ of degree 
$j$ (for a proof, see \cite[Proposition 1.7]{Eisenbud2005}).
\end{rmk}

\begin{df} \label{t0}

The largest degree among the elements of a minimal homogeneous generating set of 
$I$ is an invariant of $R/I$, defined in the following way
 $$t_1(R/I)=\max \{j \mid\beta_{1,j}(R/I)\neq0\}.$$
\end{df}

A more famous and easier invariant to compute is the \textbf{Castelnuovo--Mumford regularity}
\[\operatorname{reg}(R/I)
=
\max \{\, j - i \mid \beta_{i,j}(R/I) \neq 0 \,\}.
\]

It follows immediately from the definitions that $t_1(R/I) \leq \operatorname{reg}(R/I) +1$.

\subsection{Castelnuovo--Mumford regularity for sheaves}
\begin{df}
Let $\mathcal{F}$ be a coherent sheaf on $\mathbb{P}^n$. We say that $\mathcal{F}$ is
\textbf{$m$-regular} if
\[
H^i\bigl(\mathbb{P}^n,\mathcal{F}(m-i)\bigr)=0
\qquad\text{for all } i>0.
\]
The \textbf{Castelnuovo--Mumford regularity} of $\mathcal{F}$ is defined by
\[
\operatorname{reg}(\mathcal{F})
=
\min\left\{
m\in\mathbb{Z}
\;\middle|\;
H^i\bigl(\mathbb{P}^n,\mathcal{F}(m-i)\bigr)=0
\text{ for all } i>0
\right\}.
\]
\end{df}
The following proposition is an exercise of \cite[Chapter 4E]{Eisenbud2005}.
\begin{prop} \label{reg0} \cite[4E, Exercise 3]{Eisenbud2005}
    Suppose $X \subset \mathbb{P}^N$ is a projective scheme and $N> 0$. Then,
    $$\operatorname{reg}(\mathcal{I}_X) =
\operatorname{reg}(I_X) = 1+\operatorname{reg}(\mathcal{O}_X) = 1+ \operatorname{reg}(S(X)).$$
\end{prop}

\section{Lifting graded algebra automorphisms}
\label{3}
 In this section and throughout the rest of the paper paper, we denote by
$\mathbb{C}[T_0, \dots, T_N]$ the ring of polynomials with standard grading. Considering a finitely generated graded algebra
$L \simeq \mathbb{C}[T_0, \dots, T_N]/I$, and $f$ an homogeneous element of degree $d$,
our aim in this section is to study the
lifting of graded automorphisms of $L/(f)$ to $L$.

\begin{nt}Let $A$ be a finitely generated graded $\CC$-algebra. We denote by 
$\Autgr(A)$ the group of isomorphisms of $A$ as a $\CC$-graded algebra. Moreover, given  a graded ideal $I \subseteq A$, we denote by

$$\Autgr(A, A/I):= \{\phi' \in \Autgr(A) \mid \phi'(I)=I\}.$$
\end{nt}

\begin{nt}Let $A$ and $B$ be a graded rings and $\alpha: A \to B$ a graded homomorphism.
We denote by $\alpha_m \colon A_m \to B_m$, the degree $m$ component of $\alpha$.

\end{nt}

\begin{lem} \label{lem2} Let $A$ be a finitely generated graded $\CC$-algebra and $I \subset A$ a graded ideal. If every graded automorphism of $A/I$ lifts uniquely to a graded automorphism of $A$,
then
$$\Autgr(A/I) \simeq \Autgr(A, A/I).$$
    
\end{lem}
\begin{proof}
 We define the map \begin{align*}
\alpha \colon \Autgr(A/I) &\to \Autgr(A,A/I) \\
    \phi &\mapsto \alpha (\phi)
\end{align*}
where $\alpha (\phi)$ is the lift of $\phi$. From the hypotheses, one can verify that the map $\alpha$ is a well-defined surjective group homomorphism. Denoting by $\kappa \colon A \to A/I$ the quotient homomorphism, the lifting property implies that $ \phi \circ \kappa = \kappa \circ \alpha(\phi) $. Since $\kappa$ is surjective, every element of $A/I$ can be written as $\kappa(g)$ for some $g$ in $A$, therefore if $\alpha(\phi)= id_{A}$, we have $\phi( \kappa(g))= \kappa(g)$ for every $g \in A$, which implies \(\phi=\operatorname{id}_{A/I}\). Hence,
$\ker\alpha=\{\operatorname{id}_{A/I}\}$ and we conclude our proof.

\end{proof}
\begin{prop}

\label{prop:lift1}
Let $L := \mathbb{C}[T_0,\ldots,T_N]/I$ be a finitely generated graded algebra with $\beta_{1,1}(L)=0$. Then, any graded algebra homomorphism $\phi \colon L \to L$ admits a unique lift to a homomorphism $\phi' \colon \mathbb{C}[T_0,\dots,T_N] \to \mathbb{C}[T_0,\dots,T_N].$ Hence, $$\Autgr(\mathbb{C}[T_0,\dots,T_N], L) \simeq \Autgr(L).$$
\end{prop}

\begin{proof}
Let $\kappa \colon \CC[T_0, \dots T_N] \to L$ be the graded quotient homomorphism. We recall that the condition $\beta_{1,1}(L)=0$ implies that the degree-1 component of $I$ is zero, i.e. 
$I_1=0$, hence the degree one component of $\kappa$, $\kappa_1 \colon \CC[T_0, \dots T_N]_1 \to L_1$ is an isomorphism.
Then, for each variable $T_i$, we can define $$\phi'(T_i)=: \kappa_{1}^{-1} ( \phi_{1} ( \kappa_{1} (T_i)))\in \CC[T_0, \dots, T_N]_{1}.$$
The universal property of polynomial rings then gives a unique extension $\phi'$. If $\phi$ is an isomorphism, the same applies to its inverse, so $\phi'$ is an isomorphism. The last assertion is a direct consequence of Lemma \ref{lem2}.
\end{proof}

The idea of the next proposition is simple, first we lift $\phi \colon L/(f) \to L/(f)$ to the ambient polynomial ring using Proposition \ref{prop:lift1}; then we show this lift descends to 
$L$ by checking it kills the generators of 
$I$ using that
$L_{d_j} \xrightarrow{\pi_{d_j}} (L/(f))_{d_j}$ is an isomorphism when $d_j < \deg f$.

\begin{prop} \label{prop:lift2}
Let $L:=\mathbb{C}[T_0,\ldots,T_N]/I$ be a finitely generated graded algebra with $\beta_{1,1}(L)=0$.  Let \(f \in L\) be a homogeneous element of degree $d$ with
$ d > t_1(L)$.  Then, any homomorphism of graded algebras $\phi \colon L/(f) \to L/(f)$ admits a unique
lift to a homomorphism of graded algebras $ \phi' \colon L \to L$.
Therefore, $$\Autgr( \CC[T_0, \dots, T_N], L/f) \simeq \Autgr(L, L/f) \simeq \Autgr(L/f).$$

\end{prop}
\begin{proof}
    Let $\phi \colon L/(f) \to L/(f)$ be a graded homomorphism, since
    $L_1 \simeq (L/(f))_1$, $\beta_{1,1}(L/(f))=\beta_{1,1}(L)=0$ and then by Proposition \ref{prop:lift1}, $\phi$ lifts to a graded homomorphism $\tilde{\phi} \colon \CC[T_0, \dots, T_N] \to \CC[T_0, \dots, T_N] $. Let $\kappa \colon \CC[T_0, \dots, T_N]  \to L $ and $q \colon \CC[T_0, \dots, T_N]  \to L/(f) $ be the quotient maps. By construction,
$ q \circ \tilde{\phi} = \phi \circ q.$

We denote by $\{f_1,\ldots,f_k\}$ a set of minimal homogeneous generators of
$I$ and denote by $d_j:=\deg f_j$.
We claim that for every $f_j \in \{f_1,\ldots,f_k\}$, $\kappa(\tilde{\phi}(f_j)) =0 $. Indeed, since $d > t_1(L) \geq d_j$, we have that $(f)_{d_j}=0$ and so
the graded pieces $L_{d_j}$ and $\big(L/(f)\big)_{d_j}$ are canonically isomorphic. Hence, identifying these two graded pieces $$\kappa(\tilde{\phi}(f_j)) = \kappa_{d_j}(\tilde{\phi}_{d_j}(f_j))= q_{d_j}(\tilde{\phi}_{d_j}(f_j)))= \phi_{d_j} (q_{d_j}(f_j))=0.$$
Since $f_j$ generate $I$ as an ideal and $\kappa \circ \tilde{\phi}$ is a ring homomorphism, $ I \subseteq \ker \kappa \circ \tilde{\phi} $. 

Then by the universal property of quotient graded algebras, there is a unique graded algebra homomorphism $\phi' \colon L \to L$ such that $\phi'\circ \kappa=\kappa \circ \tilde{\phi}$. If $\phi$ is an isomorphism, the same reasoning applies to its inverse.  Finally, any automorphism of graded algebras $\phi \colon L/(f) \to L/(f)$
lifts uniquely to a automorphism of graded algebras $ \phi' \colon L \to L$, and the last assertion of the proposition follows from Lemma \ref{lem2} and Proposition \ref{prop:lift1}.

\end{proof}

\section{Automorphisms of polarized varieties and section rings}
\label{4}
The main result of this section (Proposition \ref{prop:gradedaut_autY}) states that 
for a normally generated line
bundle $B$, we have the following exact sequence

\[
1 \longrightarrow \mathbb{C}^*
\xrightarrow{\;\iota\;}
\operatorname{Aut}_{\mathrm{gr}}(R(Y,B))
\xrightarrow{\;\Phi_{(Y,B)}\;}
\operatorname{Aut}(Y,B)
\longrightarrow 1.
\]

The above result is closely related to the exact sequence of \cite[Proposition 2.13]{Brion}, although the central objects
in the two sequences are defined differently. We give a self-contained proof in the present setting
inspired by the argument of Amrani for weighted projective spaces
in \cite[Section 8]{Amrani1989}.

\begin{df}
 Let $B$ be an ample line bundle on a projective variety $Y$. We denote by
 $\Aut(Y,B)= \{g \in \Aut(Y) \mid g^*B \simeq B \}$.

\end{df}
We recall a basic algebraic geometry fact before proving the next proposition.
\begin{rmk}
\label{hartshorne}
\cite [Exercise 2.14 b, Chapter II]{Hartshorne}, Given a homomorphism $\varphi \colon S \to S'$ of graded rings, denoting by $\mathfrak{m}$ the irrelevant ideal of $S$, that is, $\mathfrak{m}= \bigoplus_{d>0} S_d$, we have that $\varphi$ induces a homomorphism
$g \colon U \to \text{Proj } S$,
where $U= \{ p \in \text{Proj } S'| \phi(\mathfrak{m}) \not\subseteq p \}$.
\end{rmk}

For our next proposition, we recall that in this paper, when refer to a variety, we refer to an integral scheme of finite type over $\mathbb{C}$.

\begin{df}
    We recall that $Y \subset \mathbb{P}^N$ is said to be nondegenerate if it $Y$ is not contained in any hyperplane of $\mathbb{P}^N$, or equivalently not contained in any proper linear subspace.
\end{df}

\begin{lem} \label{helpfullemma}
    Let $Y \subset \mathbb{P}^N$ be a nondegenerate projective variety. If $g\in \operatorname{PGL}(N+1)$ satisfies $g|_Y=\operatorname{id}_Y$,
    then $g= \operatorname{id}_{\mathbb{P}^N}$.
\end{lem}

\begin{proof}
     Let $C(Y)$ be the affine cone over $Y$ and $A$ be a lift of $g$ to $\operatorname{GL}(N+1)$. Since $g$ restricts to the identity on $Y$, for every nonzero $v \in C(Y)$  we have $$Av =    \lambda_{v} v  \text{ for some $ \lambda_{v} \in \mathbb{C}^*$}. $$
     Suppose that $A$ has $r$ distinct eigenvalues $\lambda_1, \dots, \lambda_r$ and denote by $E_{\lambda_i}$ the corresponding eigenspaces. Thus, $C(Y) \subseteq  \bigcup_{i=1}^{r}  E_{\lambda_i}$ and consequently $$ Y \subseteq \bigcup_{i=1}^{r}\mathbb{P}(E_{\lambda_i}).$$

By the convention adopted in this paper, projective varieties are irreducible, so it follows that $Y \subseteq \mathbb{P}(E_{\lambda_i}) $ for one $i \in
\{1, \dots,r\}$. Moreover, since $Y$ is nondegenerate and  $\mathbb{P}(E_{\lambda_i})$ is a linear subset of $\mathbb{P}^N$, we have that
     $\mathbb{P}(E_{\lambda_i})=\mathbb{P}^N$. Finally, we conclude that $A= \lambda_i \operatorname{id}$, and hence $g= \operatorname{id}_{\mathbb{P}^N}$.
\end{proof}

\begin{rmk} \label{C}

We note that when $S$ is a graded algebra generated by degree 1 elements,
there is a natural injective homomorphism $$ \iota \colon\mathbb{C}^* \hookrightarrow \Autgr(S),$$ given by $\iota(\lambda)|_{S_m}= \lambda^m \operatorname{id}_{S_m}, \text{ } \lambda \in \mathbb{C}^*, m \geq 0.$
\end{rmk}
\begin{prop} \label{prop:gradedaut_autY}
Let $Y$ be a projective variety and $B$ a normally generated line bundle. 
Let 
\[
R(Y,B) := \bigoplus_{m \geq 0} H^0\!\left(Y, B^{\otimes m}\right)
\]
be the section ring of $B$, and denote by $\Autgr(R(Y,B))$ the group of automorphisms of $R(Y,B)$ as a $\CC$-graded algebra. 
Then, there is a natural short exact sequence

\[
1 \longrightarrow \mathbb{C}^*
\xrightarrow{\;\iota\;}
\operatorname{Aut}_{\mathrm{gr}}(R(Y,B))
\xrightarrow{\;\Phi_{(Y,B)}\;}
\operatorname{Aut}(Y,B)
\longrightarrow 1.
\]

\noindent where $\Phi_{(Y,B)}$ is induced by the action of a graded automorphism on $\Proj R(Y,B)$.
In particular, $$\Aut(Y,B) \simeq \Autgr(R(Y,B))/\CC^*.$$
\end{prop}

\begin{proof}
Since $B$ is ample, we have $\Proj R(Y,B) \simeq Y$  by \cite[Lemma 1.1.13]{brion2005frobenius}. Therefore, by Remark \ref{hartshorne}, 
we can define a group homomorphism 
\begin{align*}
\Phi_{(Y,B)} \colon \Autgr(R(Y,B))  &\to \Aut(Y,B) \\
\varphi &\mapsto g := \Proj (\varphi)^{-1}
\end{align*} 

In order to simplify our notation, we denote 
$\Phi_{(Y,B)}$ by $\Phi$ and  $R(Y,B)$ by $R$ in the rest of the proof.
Since $B$ is normally generated, by Remark \ref{rmk1}, $R$ is generated in degree 1. We have that in this case $B \simeq \OO_Y(1)$ and therefore by \cite[Proposition 5.12c]{Hartshorne}, for every $g \in \im \Phi$,
$$g^*B \simeq B,$$
so $\im \Phi \subset \Aut(Y,B).$
Thus, in order to obtain our result, we have to show that $\im \Phi = \Aut(Y,B)$ and that
$\ker \Phi = \mathbb{C}^*.$

We start by proving $\im \Phi = \Aut(Y,B)$. Let $g \in \Aut(Y,B)$ and put $h := g^{-1}\in \Aut(Y,B)$. Fix an isomorphism
$\alpha_1 \colon h^*B \xrightarrow{\simeq} B$ and let
$\alpha_m := \alpha_1^{\otimes m} \colon h^*(B^{\otimes m}) \simeq (h^*B)^{\otimes m} \to B^{\otimes m}$.
For $s \in H^0(Y,B^{\otimes m})$ define
$$\varphi_m(s) := \alpha_m(h^*s) \in H^0(Y,B^{\otimes m}).$$
Each $\varphi_m$ is $\mathbb{C}$-linear and bijective, since $h^*$ is bijective on global
sections ($h$ is an automorphism) and $\alpha_m$ is an isomorphism. Moreover, for
$s \in H^0(Y,B^{\otimes m})$ and $t \in H^0(Y,B^{\otimes n})$ we have $h^*(st)=h^*s\cdot h^*t$
and $\alpha_{m+n}=\alpha_m\otimes\alpha_n$, hence
$$\varphi_{m+n}(st)=\varphi_m(s)\,\varphi_n(t),$$
where products are taken via the canonical multiplication maps
$H^0(Y,B^{\otimes m})\otimes H^0(Y,B^{\otimes n})\to H^0(Y,B^{\otimes(m+n)})$.
Thus $\varphi:=\bigoplus_m \varphi_m$ is a graded $\mathbb{C}$-algebra automorphism of $R$.

We claim $\Proj(\varphi)=h$. Since $Y$ is a variety over $\mathbb{C}$, it suffices to
compare the two morphisms on closed points. For a closed point $y\in Y$ let
$\mathfrak{p}_y=\bigoplus_{m\ge 0}\{s\in H^0(Y,B^{\otimes m}) : s(y)=0\}$; this is a
homogeneous prime ideal not containing $R_+$ (as $B$ is globally generated), and
$y\leftrightarrow\mathfrak{p}_y$ under $\Proj R\simeq Y$. Because $\alpha_m$ is an
isomorphism of line bundles, $\alpha_m(h^*s)$ vanishes at $y$ if and only if $h^*s$
does, i.e.\ if and only if $s(h(y))=0$. Therefore
$$\varphi^{-1}(\mathfrak{p}_y)=\mathfrak{p}_{h(y)},$$
that is, $\Proj(\varphi)(y)=h(y)$ for every closed point $y$. Since $Y$ is reduced,
separated and of finite type over $\mathbb{C}$, by \cite[II Ex 4.2]{Hartshorne}, two morphisms $Y\to Y$ that agree on
closed points are equal, so $\Proj(\varphi)=h$. Hence
$\Phi(\varphi)=\Proj(\varphi)^{-1}=h^{-1}=g$, and $g\in\im\Phi$.
This proves $\im\Phi=\Aut(Y,B)$.

We now prove that $\ker \Phi = \mathbb{C}^*.$
It is clear that $\mathbb{C}^* \subseteq \ker \Phi$. For the reverse inclusion,  recall that since $B$ is normally generated, $B$ is very ample by Remark \ref{ngisvample} and hence $B$ induces a nondegenerate embedding $i_B \colon Y \hookrightarrow \mathbb{P}^N:= \mathbb{P}(H^0(Y,B))$. Let $\varphi\in\ker\Phi$. Then $\Phi(\varphi)=\operatorname{id}$ and the linear map $\varphi_1 \colon H^0(Y,B) \to H^0(Y,B) $ induces an automorphism $g$ on $\mathbb{P}^N$ such that
$$ \operatorname{id}_Y= \Phi(\varphi)=g|_Y$$

By Lemma \ref{helpfullemma}, $g= \operatorname{id}_{\mathbb{P}^N}$ and hence
$\varphi_1= \lambda \operatorname{id}_{R_1}$ for some $\lambda \in \mathbb{C}^*$.
Since $R$ is generated by elements of degree 1, it follows that
$\varphi= \iota (\lambda)$ as in Remark \ref{C}. Thus, $\ker \Phi = \mathbb{C}^*.$

\end{proof}

\section{Varieties with an $\Aut(X)$-invariant polarization} \label{5}
In this section, we will be interested in varieties that admit an ample line bundle
$H$ such that $\Aut(X)= \Aut(X,H)$. The main motivation for studying such varieties is that their automorphism groups can be studied algebraically, as we will see in Corollary
\ref{cor:aut=autgr}.

\begin{df} Let $X$ be a projective variety and $H$ an ample line bundle on $X$. If for every $g \in \Aut(X)$, $g^*H \simeq H$, we say that $H$ is an $\Aut(X)$-invariant polarization.
In this case, we say that $X$ admits an
$\Aut(X)$-invariant polarization.
\end{df}

An immediate corollary is the following

\begin{cor} \label{cor:aut=autgr}
Let $X$ be a projective variety and $H$ an $\Aut(X)$-invariant and normally generated line bundle. We denote by $R(X,H)$ the ring of sections of $H$. Then,

$$ \Aut(X) \simeq \Autgr( R(X,H))/\mathbb{C}^*.$$
    
\end{cor}
\begin{proof}
 From the definitions it follows that if $H$ is a $\Aut(X)$-invariant polarization, $\Aut(X)= \Aut(X,H)$. Then, the result is an immediate consequence of Proposition \ref{prop:gradedaut_autY}.
\end{proof}

In what follows, we see some conditions on varieties $X$ that guarantee the existence of an $\Aut(X)$-invariant polarization $H$. Our motivation in this article is understanding when automorphisms of a hypersurface are completely determined by the ambient space. Therefore, we will be particularly interested to understand when $H|_X$ is a $\Aut(X)$-invariant polarization in following setting:
$F$ a projective variety with an ample line bundle $H$ and $X \in |H^{\otimes d} |$ for some $d >0$.

\begin{prop}\label{prop:torsionfree}
Let $X$ be a projective variety with torsion-free Picard group, and $H$ be an ample line bundle on $X$. Suppose that $\omega_X= H^{\otimes a}$ for some non-zero integer $a$, then $H$ is $\Aut(X)$-invariant.  
\end{prop}
\begin{proof}
    Let $g\in\Aut(X)$ and suppose that $g^*H\not\simeq H$.
Set $L := g^*H\otimes H^{-1}.$
Then $L$ is a nontrivial element of $\Pic(X)$. However,
\[
L^{\otimes a}
= g^*H^{\otimes a}\otimes H^{-\otimes a}
= g^*\omega_X \otimes \omega_X^{-1}
\simeq \OO_X,
\]
since automorphisms preserve the canonical bundle.
Thus $L$ is a torsion element of $\Pic(X)$, which contradicts the fact
that $\Pic(X)$ is torsion-free. Therefore
$
g^*H\simeq H
$
for every $g\in\Aut(X)$.
\end{proof}

We recall some results that will be helpful.

\begin{thm} \label{homotopy}\cite[Theorem 3.1.21]{Lazarsfeld}(Lefschetz theorem for homotopy groups). Let $F$ be a non-singular irreducible complex projective variety of dimension $n$, and let $X$ be
 any effective ample divisor on $F$. Then
 $\pi_i(F,X) =0 \text{ for } i<n.$
 In particular, the homomorphism
 $\pi_i(X) \to \pi_i(F)$
 induced by inclusion is bijective if $i \leq n-2$ and surjective if $i = n -1$.
    
\end{thm}

\begin{lem}\cite[Lemma 2.5]{EsserAutoWPS} \label{torsionfree}
Let $V$ be a connected scheme of finite type over $\mathbb{C}$. If $\pi_1(V)=1$, then $\Pic(V)$ is
 torsion-free.
\end{lem}

As a consequence, we obtain:

\begin{prop}\label{prop:Hautinvariant}
Let $F$ be a smooth simply connected projective variety with $\dim F \ge 3$.
Suppose $H$ be an ample line bundle on $F$ such that $\omega_F=H^{\otimes r}$ for some integer
$r$. Let $X\in |H^{\otimes d}|$ be a smooth element, for $d>0$ and $d \neq -r$. Then, $H|_X$ is
$\Aut(X)$-invariant.
\end{prop}

\begin{proof}
By the adjunction formula, we have $\omega_X=H|^{\otimes r+d}_X$ and since
$d \neq -r$, $r+d \neq 0$. We show that $\Pic(X)$ is torsion free.
Since $F$ is simply connected, we have $\pi_1(F)=1$. As $\dim F\ge3$,
Theorem \ref{homotopy} implies that $\pi_1(X)=1$. Moreover, we have the following exact sequence
$$0 \to H^{0}(F,H^{\otimes -d}) ) \to H^{0}(F,\OO_F ) \to H^{0}(X,\OO_X) \to H^{1}(F,H^{\otimes -d})$$
Since $H^{\otimes d}$ is ample, $H^{0}(F,H^{\otimes -d}) )=0$ and we can apply Kodaira vanishing theorem and Serre's duality to obtain $$0=h^{n-1}(F, H^{\otimes d} \otimes \omega_F) = h^1(F, H^{\otimes -d})$$

Hence, 
$H^0(X, \OO_X)\simeq H^0(F,\OO_F)=\CC$, so $X$ is connected. By Lemma \ref{torsionfree}, the group $\Pic(X)$ is torsion free and we can apply Proposition \ref{prop:torsionfree} to conclude our result.

\end{proof}

We see what happens when $F$ has mild singularities. We start recalling a version of adjunction formula in more general settings.

\begin{thm}\cite[Proposition 5.73]{KollarMori1998} \label{Gorenstein} Let $F$ be a Cohen-Macaulay projective variety  and $X \subset F$  an effective Cartier divisor. Then the adjunction formula
\[
\omega_X \simeq
\left(\omega_F \otimes \mathcal{O}_F(X)\right)\big|_X
\ \text { holds.}  \]
   
\end{thm}

We recall a generalization of the Grothendieck–Lefschetz theorem due to Ravindra and Srinivas, in a form simplified for our purposes.

\begin{thm} \label{thm:GeneralizedLefschetz} \cite[Theorem 1]{RavindraSrinivas2005}
Let $F$ be a normal projective variety %regular in codimension $1$. 
Let $H$ be an ample and base-point-free line bundle on $F$. Let $X \in |H|$ be a general member. Then the restriction map
\[
\operatorname{Cl}(F) \longrightarrow \operatorname{Cl}(X)
\]
is an isomorphism if $\dim F \geq 4$. If $\dim F=3$, then the
restriction map is injective and has finitely generated cokernel.
\end{thm}

\begin{prop}
    \label{prop:singhigherdim}

Let $F$ be a normal Gorenstein projective variety %regular in codimension 1 
and assume that $\operatorname{Cl}(F)$ is torsion-free and $\dim F \geq 4$.
Let $H$ be ample and base-point-free line bundle on $F$ such that $\omega_F \simeq H^{\otimes r}$
for some $r\in\mathbb Z$. Let $X\in |H^{\otimes d}|$ be a  general member, where
$d>0$ and $d\neq -r$. Then, $H|_X$ is
$\Aut(X)$-invariant.
\end{prop}
\begin{proof}
Since $F$ is Gorenstein and therefore Cohen-Macaulay, we can apply Theorem \ref{Gorenstein} to obtain $\omega_X=H|^{\otimes r+d}_X$. Then, the result follows from Theorem \ref{thm:GeneralizedLefschetz} and
Proposition \ref{prop:torsionfree}.
\end{proof}

When $\dim F=3$, Bruzzo, Grassi and Lopez obtained that the 
map  $\operatorname{Cl}(F) \to \operatorname{Cl}(X) $ is in fact an isomorphism under more restrictive hypotheses.

\begin{thm}\label{AAU}\cite[Proposition 3.2]{BruzzoGrassiLopez2021}, based on \cite{RavindraSrinivas2009}. Let $F$ be a normal  $\mathbb{Q}$-factorial threefold with rational singularities, and let $H$ be a very ample line bundle on $F$. 
Assume that $\omega_F \otimes H$ is globally generated. Then:

\begin{enumerate}
    \item $\mathrm{Cl}(F) \simeq \mathrm{Cl}(X)$ for a very general $X \in |H|$.
   \item $\rho(X) = \rho(F)$ for a very general $X \in |H|$.
\end{enumerate}
    
\end{thm}

As a consequence, we obtain the following corollary.

\begin{prop} \label{prop:singular3fold} Let $F$ be a normal, $\mathbb{Q}$-factorial Gorenstein threefold with rational singularities, and assume that $\operatorname{Cl}(F)$ is torsion-free. Let $H$ be a very ample line bundle on $F$ such that $\omega_F \simeq H^{\otimes r}$
for some $r\in \mathbb{Z}$. Let $X\in |H^{\otimes d}|$ be a very general member, where
$d>0$ and $d+r>0$. Then, $H|_X$ is
$\Aut(X)$-invariant.
    
\end{prop}

\begin{proof} Since $H$ is very ample, 
 $H^{\otimes (r +d)}=\omega_F \otimes H^{\otimes d}$ is globally generated. Then the assertion follows from Theorem \ref{Gorenstein}, Theorem \ref{AAU} and Proposition \ref{prop:torsionfree}.
\end{proof}

\section{Extending automorphisms via section rings}
\label{sec.Extaut}
In this section, we present the main result of this paper, Theorem \ref{thm:main}. Let $F$ be a projective variety and $H$ a normally generated line bundle on $F$. Given $X\in |H^{\otimes d}|$, we establish conditions under which every automorphism of $X$ is induced by an automorphism of $F$. We then derive several corollaries providing more concrete settings in which these hypotheses are satisfied. Before proving the main result, we first establish conditions under which the section ring $R(X,H|_X)$ can be described in terms of $R(F,H)$. This will be a fundamental step in the proof of our main theorem.

\begin{prop} \label{coxring}
Let $F$ be a projective variety, $H$ an ample line bundle, and
$R(F,H)$ its section ring. For a homogeneous nonzero divisor
$f\in R(F,H)_d \setminus\{0\}$, let $X=V(f)\subset F$. If
$H^1(F,H^{\otimes(m-d)})=0$ for all $m\geq0$, then
$R(F,H)/(f)\simeq R(X, H|_X)$
as graded rings.
\end{prop}

\begin{proof}
In order to simplify notation,
we fix $R:=R(F,H)$. Since $f$ is a nonzero divisor, we have the following exact sequence of graded modules:

        $$0 \rightarrow R(-d) \xrightarrow{\cdot f} R\rightarrow R/(f)\rightarrow 0. $$

Since $R(m)$ is a free module of rank 1 over $R$, therefore flat,
we obtain the following exact sequence:

  $$0 \rightarrow R(m-d) \xrightarrow{\cdot f} R (m)\rightarrow R/(f) (m)\rightarrow 0. $$

Denoting by $R_m$ the group of homogeneous elements of $R$ of degree
$m$ and noticing that $R(m)_0=R_m$
and using the analogous equalities and notation for the other rings, we have the following exact sequence, considering the degree 0 components,

  $$0 \rightarrow R_{m-d} \xrightarrow{\cdot f} R_{m} \rightarrow (R/(f))_m\rightarrow 0. $$

After tensorizing the standard exact sequence associated to the divisor $X$, we also obtain
the following exact sequence: $$0 \rightarrow H^{\otimes m-d} \rightarrow H^{\otimes m} \rightarrow (H|_X)^{\otimes m} \rightarrow 0. $$

From the hypotheses, for all $m \geq 0$ we have the following diagram:
\begin{center}

\begin{tikzcd}
0 \arrow[r] \arrow[d] &  R_{m-d} \arrow[r, "\cdot f"] \arrow[d,"="]              &  R_{m} \arrow[r] \arrow[d,"="]                &  (R/(f))_m \arrow[r]              & 0 \arrow[d] \\
0 \arrow[r]           & {H^{0}(F,H^{\otimes m - d} )} \arrow[r,"\cdot f" ] & {H^{0}(F,H^{\otimes m} )} \arrow[r] & {H^{0}(X,(H|_X)^{\otimes m})} \arrow[r] & 0          
\end{tikzcd}
\end{center}

Thus, $(R/(f))_m \simeq R_m / \ker (\cdot f) \simeq  H^{0}(X, (H|_X)^{\otimes m})$, concluding our result.

\end{proof}

\begin{thm}\cite[Thm 7.80]{SchiffmanSommese} \label{vanishing} Let $F$ be a normal  projective algebraic variety and let $H$ be an ample invertible sheaf. Then, \begin{enumerate}

   \item  IF $F$ is Cohen-Macaulay $H^b
(F, \omega_F \otimes H) = 0 \text{, for }b > \dim \Sing (F)$.
    
\item If $F$ has rational singularities, then $H^b
(F, \omega_F \otimes H) = 0 \text{ for } b > 0$.

\end{enumerate}
\end{thm}

We recall the well-known Serre duality theorem.

\begin{thm}\cite[Corollary 7.7]{Hartshorne} \label{duality}
    Let $F$ be a projective Cohen-Macaulay variety. Then for any locally free sheaf $\mathcal{L}$ on $F$, there are natural isomorphisms
$$ H^i(F, \mathcal{L})^ \vee\simeq H^{n-i}( F, \mathcal{L}^\vee \otimes \omega_F).  $$
\end{thm}

In the next proposition, we see some sufficient conditions for the required vanishing in Proposition \ref{coxring}.

\begin{prop} \label{h1vanish}
Let $F$ be a Gorenstein, normal projective variety and $\dim F\geq2$. Assume that either $F$ has only isolated singularities or $F$ has rational singularities. Let
$H$ be an ample line bundle on $F$. 
\begin{enumerate}
\item $\omega_F\simeq\mathcal{O}_F$ and $H^1(F, O_F) = 0$ , or
\item $F$ is a Fano and $\omega_F\simeq H^{\otimes r}$ for some integer $r<0$.
\end{enumerate}

Then, for any positive integer $d$, $H^1(F,H^{\otimes(m-d)})=0$ for all $m\geq0$.
\end{prop}

\begin{proof}
   
Let $\omega_F\simeq H^{\otimes r}$, with $r \leq 0$.
Suppose first that $m-d>r$. Writing $m-d=n+r$ with $n>0$, we have $H^{\otimes(m-d)}\simeq \omega_F\otimes H^{\otimes n}.$
Since $H^{\otimes n}$ is ample, Theorem \ref{vanishing} gives $H^1(F,H^{\otimes(m-d)})=0.$
Suppose next that $m-d<r$. Then $d-m>-r\geq0$,
so $H^{\otimes(d-m)}$ is ample. By Theorems \ref{duality} and \ref{vanishing},$$H^1(F,H^{\otimes(m-d)})^\vee
\simeq H^{N-1}\bigl(F,\omega_F\otimes H^{\otimes(d-m)}\bigr)
=0,$$
\noindent where $N=\dim F$. Finally, suppose that $m-d=r$. If $F$ is Fano, then $r<0$, so
$H^{\otimes(d-m)}=H^{\otimes(-r)}$ is ample, and the same argument applies. If $\omega_F\simeq\mathcal O_F$,
then $r=0$, and hence $H^1(F,H^{\otimes(m-d)})=H^1(F,\mathcal O_F)=0.$

\end{proof}

\begin{nt}
Given $X\xhookrightarrow{i} F$ projective varieties, and $H$ an ample line bundle on $F$, we denote by $$\Aut(F;X):=\{g \in \Aut(F) \mid g(X) = X \}$$
the group of automorphisms of $F$
preserving $X$.
\end{nt}
\begin{lem}\label{lem:autFXFH}
Let $F$ be a normal projective variety with $\dim F\geq 3$, and let
$H$ be an ample and base-point-free line bundle on $F$. Let
$X\in |H^{\otimes d}|$ be a general projective variety. Assume that $H|_X$ is
$\Aut(X)$-invariant. Then $$\Aut(F;X)\subseteq \Aut(F,H).$$
The same conclusion holds if $F$ is smooth and $X$ is a smooth
member of $|H^{\otimes d}|$.
\end{lem}

\begin{proof}
   Let $g\in \Aut(F;X)$,
 $X \xhookrightarrow{i} F$ be the inclusion, then $g|_{X}^*(i^*H) \simeq i^*g^*H.$ Since $i^*H := H|_X$, and $g|_X^*(H|_X) \simeq H|_X $, we have $i^*g^*H \simeq i^*H.$
As $\dim F \ge 3$, the map $\operatorname{Cl}(F) \xhookrightarrow{i^*} \operatorname{Cl}(X)$ is injective by Theorem \ref{thm:GeneralizedLefschetz},
hence $g^*H \simeq H $. In the smooth case, we use the classical form of Lefschetz theorem to obtain $\operatorname{Pic}(F) \xhookrightarrow{i^*} \operatorname{Pic}(X)$ is injective.
\end{proof}

\begin{rmk} We defined $t_1(A)$ for a finitely generated $\mathbb{C}$-algebra $A$ in Section \ref{2}.
Here, we stablish the convention that when
$A= \mathbb{C}[T_0, \dots, T_N]$, we set $t_1(\mathbb{C}[T_0, \dots, T_N])=1$.
\end{rmk}

We are now ready to state our main theorem.

\begin{thm} \label{thm:main}
 Let $F$ be a projective variety and $H$ normally generated line bundle on $F$ and $R(F,H)$ its section ring. Let $X \in |H^{\otimes d}|$ a projective variety, for $d >t_1(R(F,H))$. Suppose 
 \begin{enumerate}
     \item $H^1(F,H^{\otimes(m-d)})=0$ for all $m\geq0$,
     \item   $H|_X$ is $\Aut(X)$-invariant,
    
     \item $\Aut(F;X)\subseteq \Aut(F,H).$
 \end{enumerate}

 Then,
$$\Aut(X)\simeq \Aut(F;X) \simeq\Aut(\mathbb{P}^N;X)$$
where $i_H \colon F \hookrightarrow \mathbb{P}^N$ is the embedding induced by the complete linear system $|H|$.
\end{thm}
\begin{proof}
Consider $f$ to be the homogeneous element of $R(F,H)_d$ such that $X= V(f) \subset F$ and denote by $R(X, H|_X)$
the section ring of $H|_X$.
Then, by Proposition \ref{coxring},
$R(X, H|_X)\simeq R(F,H)/(f) .$ 
Since $H|_X$ is $\Aut(X)$-invariant, by Corollary \ref{cor:aut=autgr}, and normally generated by Proposition \ref{restrictionNG}, hence
\begin{equation} \label{eq3}
    \Aut(X)\simeq \frac{\Autgr\bigl(R(F,H)/(f)\bigr)}{\mathbb{C}^*}.
\end{equation}

Suppose $F \neq \mathbb{P}^N$, in this case, since $H$ is normally generated and hence very ample by Remark \ref{ngisvample}, the image of $F$ under the embedding $i_{H}\colon F \hookrightarrow\mathbb{P}^N$ is nondegenerate. Therefore, its homogeneous ideal $I_F \subset \mathbb{C}[T_0, \dots, T_N]$ contains no linear forms, consequently $\beta_{1,1}(R(F,H))=0$. Thus, Propositions \ref{prop:lift1} and \ref{prop:lift2} apply for every $f$ with degree $d > t_1(R(F,H))$ and yield: 
\begin{equation}\label{eq4}
\frac{\Autgr\bigl(R(F,H)/(f)\bigr)}{\mathbb{C}^*}\simeq \frac{\Autgr\bigl(R(F,H), R(F,H)/(f)\bigr)}{ \mathbb{C}^*}\simeq \frac{\Autgr\bigl( \mathbb{C}[T_0, \dots, T_N], R(F,H)/(f)\bigr)}{\mathbb{C}^*}.
\end{equation}

Suppose $F= \mathbb{P}^N$, in this case
$R(X, H|_X)= \mathbb{C}[T_0, \dots, T_N]/(f)$ and Proposition \ref{prop:lift1} can only be applied if
$\beta_{1,1}(R(X,H|_X))=0$, that is if $d=\deg f >1$. By our convention, this happens when $d> t_1(R(F,H))= t_1( \mathbb{C}[T_0, \dots,T_n])=1. $

$$ \text{Since }
\Aut(\mathbb{P}^N; X) = \{ g \in \operatorname{PGL}(N+1) \mid g(I_X)= I_X \}
\text{ and } R(F,H)/(f)=R(X, H|_X) \simeq \mathbb{C}[T_0, \dots,T_N]/I_X,$$
we obtain
\begin{equation} \label{eq5}
\Aut(\mathbb{P}^N;X)\simeq \dfrac{\Autgr\bigl( \mathbb{C}[T_0, \dots, T_N], R(F,H)/(f)\bigr)}{\mathbb{C}^*}.
\end{equation}

We conclude our result by proving $\dfrac{\Autgr\bigl(R(F,H), R(F,H)/(f)\bigr)}{ \mathbb{C}^*}\simeq \Aut(F;X). $
For that, we define the following map \begin{align*}
\Psi \colon \Autgr\bigl(R(F,H), R(F,H)/(f)\bigr)&\to \Aut(F,H) \\
\varphi &\mapsto g \colon \Proj R(F,H) \to \Proj R(F,H)
\end{align*} 

Since $ \Autgr\bigl(R(F,H), R(F,H)/(f)\bigr) \leq \Autgr\bigl(R(F,H) \bigr) $ and 
$\Psi$ is the restriction of  $\Phi_{(F.H)}$ (see Proposition \ref{prop:gradedaut_autY} for notation) to $\Autgr\bigl(R(F,H), R(F,H)/(f)\bigr)$, by Proposition \ref{prop:gradedaut_autY},
$\ker \Psi = \mathbb{C}^*$.

We show $\im \Psi = \Aut(F;X)$. For proving $\im \Psi \subseteq \Aut(F;X)$,
we notice that any commutative diagram\begin{center}
\begin{equation}
\label{diagram1}
 \begin{tikzcd}
R(F,H) \arrow[rr, "i^*", twoheadrightarrow] \arrow[d, "\tilde{\varphi}"] &  &  R(F,H)/(f)  \arrow[d, "\varphi"] \\
R(F,H) \arrow[rr, "i^*", twoheadrightarrow] &  & R(F,H)/(f)                      
\end{tikzcd}
\end{equation}
\end{center} 
yields the following one\begin{center}
\begin{equation}
\label{diagram2}    
\begin{tikzcd}
\text{Proj } R(F,H)/(f)= X \arrow[rr, "i", hook] \arrow[d, "\Phi_{(X,H_X)}(\varphi)"] &  & F= \text{Proj }   R(F,H)\arrow[d, "\Phi_{(F,H)}(\tilde{\varphi})"] \\
\text{Proj } R(F,H)/(f) = X \arrow[rr, "i", hook]                &  & F = \text{Proj }  R(F,H)            .       
\end{tikzcd}
\end{equation}
\end{center}

We prove now the reverse inclusion. 
Let $g \in \Aut(F;X)$, and $g|_X$ the restriction. Since $H|_X$ is $\Aut(X)$-invariant, by Proposition \ref{prop:gradedaut_autY} there is a unique (up to scalar) $$\varphi \colon R(F,H)/(f) \to R(F, H)/(f) \text{ such that } \Phi_{(X,H|_X)}^{-1}(g|_X)= \varphi.$$ 
By  Proposition \ref{prop:lift2},
there is a unique graded automorphism
$\tilde{\varphi} \colon R(F,H) \to R(F,H)$ making the diagram \eqref{diagram1} commute.

By Proposition \ref{prop:gradedaut_autY}, $ \Phi_{(F,H)}(\tilde{\varphi})$ is the unique map in $\Aut(F,H)$ making the diagram \eqref{diagram2} commute. 
Since $g \in \Aut(F,H)$
and 
$\Phi_{(X,H|_X)}(\varphi)=g|_X$, we conclude $$\Phi_{(F,H)}(\tilde{\varphi})= g.$$
Moreover, $\tilde{\varphi}\in \Autgr\bigl(R(F,H), R(F,H)/(f)\bigr)$, thus $g= \Psi(\tilde{\varphi})$ and
$\Aut(F;X) \subseteq \im \Psi$. Finally,
\begin{equation} \label{eq6}
\frac{\Autgr\bigl(R(F,H), R(F,H)/(f)\bigr)}{ \mathbb{C}^*}\simeq \Aut(F;X).\end{equation}

Combining equations \eqref{eq3}, \eqref{eq4},
\eqref{eq5} and \eqref{eq6}, we obtain the desired result.

\end{proof}

\begin{cor}
    \label{cor:main}

Let $F$ be a smooth simply connected projective variety with  $\dim F  \geq 3$ and $H$ a normally generated line bundle on $F$.  Assume that either $\omega_F\simeq\mathcal{O}_F$ or $F$ is a Fano variety and $\omega_F\simeq H^{\otimes r}$ for some integer $r<0$. Let $X \in |H^{\otimes d}|$ a smooth projective variety, for $d > t_1(R(F,H)) \text{ and } d \neq -r .$ Then, $$\Aut(X)\simeq \Aut(F;X) \simeq\Aut(\mathbb{P}^N;X),$$ 
 where $i_H \colon F \hookrightarrow \mathbb{P}^N$ is the embedding induced by the complete linear system $|H|$.
\end{cor}

\begin{proof}
By Proposition \ref{h1vanish}, $H^1(F,H^{\otimes(m-d)})=0$ for all $m\geq0$, by Proposition \ref{prop:Hautinvariant}, $H|_X$ is $\Aut(X)$-invariant and $\Aut(F;X)\subseteq \Aut(F,H)$ by Lemma \ref{lem:autFXFH}, then Theorem \ref{thm:main} applies.

\end{proof}
\begin{rmk}\label{rmk:pic1}
    Suppose that $F$ is a Fano variety with Picard number $1$ and
    $\dim F \geq 4$. As a consequence of Theorem
    \ref{thm:GeneralizedLefschetz}, $X$ also has Picard number $1$.
    Since every $g\in\Aut(X)$ induces an automorphism of
    $\Pic(X)\simeq\mathbb{Z}$ and preserves the ample cone, we have $g^*H \simeq H.$ 
    Hence, $H|_X$ is $\Aut(X)$-invariant. Therefore, in this case, the
    condition $d\neq -r$ can be omitted.
\end{rmk}

\begin{cor} \label{cor:Sing3}

Let $F$ be a normal, $\mathbb{Q}$-factorial Gorenstein projective threefold with rational singularities and  $\operatorname{Cl}(F)$ is torsion-free. Let $H$ be a normally generated line bundle on $F$. Assume that either
\begin{enumerate}
\item$\omega_F\simeq\mathcal{O}_F$ and $H^1(F, \OO_F)=0$; or
\item  $F$ is a Fano and $\omega_F\simeq H^{\otimes r}$ for some integer $r<0$.
    
\end{enumerate}

Moreover, let $X\in |H^{\otimes d}|$ be a very general projective variety, where
$d>t_1(R(F,H))$ and $d+r>0$. Then, $$\Aut(X)\simeq \Aut(F;X) \simeq\Aut(\mathbb{P}^N;X),$$ 
 where $i_H \colon F \hookrightarrow \mathbb{P}^N$ is the embedding induced by the complete linear system $|H|$.
\end{cor}

\begin{proof} As a consequence of Proposition \ref{h1vanish},  by Proposition \ref{prop:singular3fold} and   Lemma \ref{lem:autFXFH}, Theorem \ref{thm:main} applies.
 
\end{proof}

\begin{cor}
Let $F$ be a normal Gorenstein projective variety with $\operatorname{Cl}(F)$ is torsion-free and $\dim F \geq 4$. Assume that either $F$ has only isolated singularities or $F$ has rational singularities.
Let $H$ be a normally generated line bundle on $F$. Assume that either
\begin{enumerate}
\item$\omega_F\simeq\mathcal{O}_F$ and $H^1(F, \OO_F)=0$; or
\item  $F$ is a Fano and $\omega_F\simeq H^{\otimes r}$ for some integer $r<0$.
    
\end{enumerate}

Moreover, let $X\in |H^{\otimes d}|$ be a general projective variety, where
$d>t_1(R(F,H))$ and $d\neq -r$. Then, $$\Aut(X)\simeq \Aut(F;X) \simeq\Aut(\mathbb{P}^N;X)$$

where $i_H \colon F \hookrightarrow \mathbb{P}^N$ is the embedding induced by the complete linear system $|H|$.
\end{cor}
\begin{proof} As a consequence of Proposition \ref{h1vanish},  by Proposition \ref{prop:singhigherdim} and   Lemma \ref{lem:autFXFH}, Theorem \ref{thm:main} applies.
 
\end{proof}

\section{Applications}
\label{7}

Our main results in Section \ref{sec.Extaut} rely on the invariant 
$t_1(R(F,H))$. In this section, we collect cases for which this invariant was already known and establish upper bounds to it on other cases.

\subsection{Normally Presented Line Bundles}

We recall that a normally presented Line bundle $H$ on $F$ induces a projetively normal embedding $\phi_{H}\colon F \hookrightarrow \mathbb{P}^N$
($H$ is normally generated) such that $I_F$ is quadratic generated.

\begin{thm} \label{thm:NP}
Let $F$ be a smooth simply connected projective variety with  $\dim F  \geq 3$ and $H$ a normally presented line bundle on $F$.  Assume that either $\omega_F\simeq\mathcal{O}_F$ or $F$ is a Fano variety and $\omega_F\simeq H^{\otimes r}$ for some integer $r<0$. Let $X \in |H^{\otimes d}|$ a smooth element, for $d > 2 \text{ and } d \neq -r $. Then, $$\Aut(X)\simeq \Aut(F;X) \simeq\Aut(\mathbb{P}^N;X),$$ 
 where $i_H \colon F \hookrightarrow \mathbb{P}^N$ is the embedding induced by the complete linear system $|H|$.
\end{thm}
\begin{proof}
Since $H$ is normally presented, $H$ is normally generated and $t_1(R(F,H))=2$, thus the theorem is a direct consequence of Corollary \ref{cor:main}.
\end{proof}

We give some known examples of normally presented line bundles.

\begin{table}[ht]
\centering
\hspace*{-1.5cm}
\begin{tabular}{c|c|c|c} \label{N1table}
&\text{Variety }  & \text{Normally Presented Line bundle }  & \text{Reference} \\
\hline
1&$\mathbb{P}^{n_1} \times \mathbb{P}^{n_2} \times \dots
\times \mathbb{P}^{n_k}$
&
$\mathcal{O}_{\mathbb{P}}(1,\dots,1)$
&
\text{\cite[Corollary 1.5]{toric}}
\\[6pt]

2&$\mathbb{P}^{n_1} \times \mathbb{P}^{n_2} \times \dots
\times \mathbb{P}^{n_k}$
&
$\mathcal{O}_{\mathbb{P}}(a_1,\dots,a_k)$, $\min\{a_i \mid a_i \neq 0\} \geq 1$
&
\text{\cite[Corollary 1.5]{toric}}
\\[6pt]

3&$\text{Veronese variety }
\nu_d(\mathbb{P}^n)
\subset
\mathbb{P}\left(
H^0\left(\mathbb{P}^n,
\mathcal{O}_{\mathbb{P}^n}(d)\right)^\vee
\right)$

&$\mathcal{O}_{\mathbb{P}^n}(d)$
&
$\text{Classical}$
\\[6pt]

4&$\text{$n$-dimensional toric variety}$
&
$B^{\otimes m},
B\text{ ample and }m\geq n$
&
\text{\cite[Corollary 1.2]{toric}}
\\[6pt]

5&$\text{Quadric hypersurface in }\mathbb{P}^{n+1}$
&
$\mathcal{O}_{X}(1)$
&
$\text{Classical}$
\\[6pt]
6&\text{Grassmannian}
$\operatorname{Gr}(k,n)$
&
$\mathcal{O}_{\operatorname{Gr}(k,n)}(1)$
&
$\text{Classical, due to Hodge}$ 
\\[6pt]

7&$\text{Lagrangian Grassmannian}\operatorname{LG}(n,2n)$
&
$\mathcal{O}_{\operatorname{LG}(n,2n)}(1)$
&\cite[Theorem 3.11]{SchubertN1}

\\[6pt]

8&\text{Spinor Varieties }$S_n$
& $\mathcal{O}_{S_n}(1)$
&\cite[Theorem 3.11]{SchubertN1}

\\[6pt]

9&\text{Rational Homogeneous Varieties }$G/P$
& $L$ ample
&\cite[Theorem 3.11]{SchubertN1}

\end{tabular}
 
\end{table}

%\begin{rmk}

%We also mention a classical result of Ein and Lazarsfeld \cite[Theorem 1]{EinLazarsfeld} which states that if $F$ is a smooth complex projective variety of dimension $n$, and  $B$ is a very ample line bundle on $F$,then
%$$ K_F \otimes B^{\otimes m} \text{ is normally presented if } m \geq n+2.$$ Recently,
%Purnaprajna and Lacini \cite[Theorem 1.2]{adjointImproved} improved this result by replacing the very ample hypothesis on 
%$B$ with the weaker assumption that $B$ is ample and base point free. 
%\end{rmk}

\begin{rmk} Let $G$ be a simply connected semisimple algebraic group over an algebraically
closed. Let $B$ be a Borel subgroup and let $P \supset B$ a parabolic subgroup. A (generalized) Schubert variety in $G/P$ is the closure of a $B$-orbit in $G/P$.  We remark that the result \cite[Theorem 3.11]{SchubertN1} states that an ample line bundle $L$ on a Schubert variety, is normally presented. We notice that $G/P$ occurs as the maximal Schubert variety in  $G/P$ (see \cite[Example 2, Section 1.2.3]{Brion2005} ), and therefore the examples 6, 7, 8 and 9 in the table above are consequences of \cite[Theorem 3.11]{SchubertN1}.
    
\end{rmk}

\begin{thm} \label{rational}

Let $F$ be a rational homogeneous variety $G/P$ with $\Pic(G/P)= \mathbb{Z} \cdot H$ and $\dim F \geq 3$. For $d>2$, let $X \in |H^{\otimes d}|$ a smooth projective variety. If $\dim F=3$, assume also that $-\omega_F \not\simeq H^{\otimes d}$. Then, $$\Aut(X)\simeq \Aut(F;X) \simeq\Aut(\mathbb{P}^N;X),$$ 
 where $i \colon F \hookrightarrow \mathbb{P}^N$ is the embedding induced by the complete linear system $|H|$.  
\end{thm}
\begin{proof}
    Consequence of Theorem \ref{thm:NP}, Remark \ref{rmk:pic1} and \cite[Theorem 3.11]{SchubertN1}.
\end{proof}

\begin{rmk}
We remark that our result for rational homogeneous varieties complements the
result of Benedetti and Manivel \cite{BenedettiManivel}. While their
result concerns the extension of automorphisms of hyperplane sections
of homogeneous varieties, our result treats hypersurfaces of arbitrary
degree $d\geq 3$.
\end{rmk}

\subsection{Anticanonical embeddings of Fano threefolds}
For a Fano threefold $F$, when $- \omega_F$ is very ample, as a consequence of Noether-Enriques-Petri on canonical curves, we always have that $\omega_F^{-1}$ is normally generated and $t_1(R(F,\omega_F^{-1})) \leq 3$ if $F$ is not quartic in $\mathbb{P}^4$.
We state here the version of theorem stated in \cite{trigonal}, for a more detailed proof, we refer to \cite[Proposition 3.8]{anticanonical}

\begin{thm} \cite[Theorem 2.14]{trigonal}, based on \cite{anticanonical}. \label{Fano}
Let $F\subset \mathbb{P}^n$
be an anticanonically embedded Fano threefold with canonical singularities,
that is,$$\omega_F^{-1}\simeq \mathcal{O}_{\mathbb{P}^n}(1)|_F
\qquad\text{and}\qquad
n=-\frac{1}{2}K_F^{3}+2.
$$
Then the following assertions hold.

\begin{enumerate}
    \item[(1)] $\omega_F^{-1}$ is a normally generated line bundle.

    \item[(2)] If $-K_F^3=4$
    then \(F\) is a quartic threefold in \(\mathbb{P}^4\).

    \item[(3)] If $-K_F^3\geq 6$, then the graded ideal $I_F$  is generated
    by its components of degrees \(2\) and \(3\).

    \item[(4)] If$-K_F^3\geq 6$
    then the graded ideal $I_F$ of is generated
    by its quadratic component, except when
    for a general linear subspace $\Pi\subset\mathbb{P}^n$ of codimension $2$, the curve $F\cap\Pi$ is either
    \begin{enumerate}
        \item [(a)] 
    a canonically embedded smooth trigonal curve; or
    \item [(b)] a
    canonically embedded smooth plane quintic, and
    $\deg(F\subset\mathbb{P}^n)=10.$

    \end{enumerate}

\end{enumerate}
\end{thm}

\begin{df}
A Fano threefold $F$ is called trigonal, if the linear system $|-K_F|$ is very ample, but the anticanonical image $F \hookrightarrow \mathbb{P}^N$ is not a an intersection of quadrics.  
\end{df}

  There are 69 families of trigonal Fano threefolds with canonical Gorenstein singularities and they were completely classified by Przyjalkowski, Cheltsov and Shramov in \cite{trigonal}. We display here only the smooth classification. The reference for the following table is Fanography \cite{Fanography} and \cite[Chapter 6,Table 3]{Prokhorov2025}.

\begin{table}[ht]
    \centering
    \begin{tabular}{c|c|c|l}
        \hline
     \text{Family} & $\rho(F)$ & $(-K_F)^3 $& \text{Description} \\
        \hline
         1.2& 1 & 4 &
        $X_4\subset \mathbb{P}^4$ \\[2pt]
        1.3 & 1 & 6 &
        $X_{2,3}\subset \mathbb{P}^5$ \\[2pt]
       
        2.4 & 2 & 10 &
        $\operatorname{Bl}_{C}\mathbb{P}^3,\quad
        C=(3,3)\subset\mathbb{P}^3$ \\[2pt]
        2.5 & 2 & 12 &
       $ \operatorname{Bl}_{C}X_3,\quad

        C\subset X_3\text{ a plane cubic}$ \\[2pt]
3.2 & 3& 14 & \text{the proper transform of a divisor in Q of
type (3, 2)} \\[2pt]

        8.1 & 8 & 18 &
       $ \operatorname{Bl}_6 \mathbb{P}^2 \times\mathbb{P}^1$\\[2pt]
        \hline
    \end{tabular}
    \caption{Smooth trigonal Fano threefolds }
    \label{tab:trigonal-fano}
\end{table}

 We recall that $\mathbb{Q}$-factorial Gorenstein Fano threefold $F$ with terminal singularities are factorial, hence 
$\operatorname{Cl}(F)$ is torsion-free, see \cite[Lemma 5.1]{Kaw88}. We therefore obtain the following theorem. 
\begin{thm} \label{Fano}
 Let $F$ be a normal, factorial, Fano threefold with terminal singularities and $\omega_F^{-1}$ is very ample. Let $X\in |\omega_F^{\otimes -d}|$ be a projective variety.
 Assume that \begin{enumerate}
    \item  $d>2$ if $F$ is not trigonal.
 \item $d>4$ if $F$ is a quartic in $\mathbb{P}^4$.
 \item $d>3$ remaining trigonal cases.
 \end{enumerate}
 Then, denoting $i \colon F\hookrightarrow\mathbb P^N$
the anticanonical embedding, we have
\begin{enumerate}
    \item[(i)] if \(F\) is smooth and \(X\) is smooth, then
    \[
    \operatorname{Aut}(X)
    \simeq
    \operatorname{Aut}(F;X)
    \simeq
    \operatorname{Aut}(\mathbb P^N;X);
    \]
    
    \item[(ii)] if \(F\) is singular, then the same conclusion holds
    for a very general \(X\in|\omega_F^{\otimes -d}|\).
\end{enumerate}

\end{thm}

\begin{proof}
Direct consequence of Theorem \ref{Fano} and Corollary \ref{cor:main} for the smooth case and \ref{cor:Sing3} for the more general case.

\end{proof}

\subsection{Upper bounds via Castelnuovo--Mumford regularity}
In this section, we compute upper bounds to $t_1(R(F,H))$ in the case that $H$ is a normally generated line bundle on variety. We recall that given an ample and globally generated line bundle $H$ on a projective variety, we have that $H^{\otimes m}$ is always normally generated if $m \geq \dim F +1$, as a consequence of recent results of Purnaprajna and Lacini \cite{adjointImproved}.

\begin{prop} \label{reg}

Let $F$ be $n$-dimensional smooth projective variety $F$ with trivial canonical divisor and $\dim F \geq 2$. 
Let $H$ be a normally generated line bundle on $F$, then $\operatorname{reg}(R(F,H))=n+1$.
    
\end{prop}
\begin{proof}
  Since $H$ is normally generated, the homogeneous coordinate ring of $F$, $S(F)$ is equal to $R(F,H)$. Hence, by Proposition \ref{reg0}, $\operatorname{reg}(R(F,H))= \operatorname{reg} \OO_F$. By Theorem \ref{duality} (Serre's duality), $H^n(F, \OO_F) = H^0(F, \OO_F)=\mathbb{C}$, hence $\operatorname{reg} \OO_F \geq n+1$. On the other hand, $H^i(F, \OO_F(n+1-i))=0 \text{ for every $i >0$ }$ by Theorem \ref{vanishing}, concluding our assertion.

\end{proof}

\begin{thm} \label{trivialcanonical} Let $F$ be a smooth simply connected projective variety with  $n:=\dim F \geq 3$ and $H$ a normally generated line bundle on $F$.  Assume that $\omega_F\simeq\mathcal{O}_F$. For $d>n+2$, let $X \in |H^{\otimes d}|$ a smooth projective variety. Then
$$\Aut(X)\simeq \Aut(F;X) \simeq\Aut(\mathbb{P}^N;X),$$ 
 where $i \colon F \hookrightarrow \mathbb{P}^N$ is the embedding induced by the complete linear system $|H|$.  
    
\end{thm}
\begin{proof}
We recall that $t_1(R(F,H)) \leq \operatorname{reg}(R(F,H))+1 $
then it follows from Corollary \ref{cor:main}  and Proposition \ref{reg}.
\end{proof}

We note that the equality $t_1(R(F,H)) \leq \operatorname{reg}(R(F,H))+1 $ is not always sharp. An example of this behavior is our computation for a specific family of Calabi--Yau varieties,  Wehler varieties, that will be presented in the next section. Another example is given by the following result of Rios-Ortiz, Rojas, and Song \cite{ORS}. They follow the notation introduced by Debarre in \cite{Deb22}, where
$M_{2d}^{(\gamma)}$ denotes
the moduli space of polarized fourfolds $(F,H)$ of $K3^{[2]}$-type satisfying $q(H)=2d$ and divisibility $\gamma$ with respect to the  Beauville--Bogomolov--Fujiki form $q$ on $H^2(F,\mathbb{Z})$.

\begin{thm}\cite[Theorem E]{ORS}

\begin{enumerate}
    \item For a general polarized hyperk\"ahler fourfold $(F,H)\in M_4^{(1)}$, the embedding
    \[
    F\hookrightarrow
    \mathbb{P}\bigl(H^0(F,H)^\vee\bigr)=\mathbb{P}^9 \text{ is projectively normal.}
    \]
    In particular, the homogeneous ideal of $F$ in $\mathbb{P}^9$ is generated by $10$ cubics and $20$ quartics.

    \item For a general polarized hyperk\"ahler fourfold $(F,H)\in M_6^{(1)}$, the embedding \[F\hookrightarrow
    \mathbb{P}\bigl(H^0(F,H)^\vee\bigr)=\mathbb{P}^{14} \text{ is projectively normal.}
    \]
    The homogeneous ideal of $F$ in $\mathbb{P}^{14}$ is generated by $15$ quadrics and at most $55$ cubics. 
\end{enumerate}

\end{thm}
\begin{cor} \label{hyperkahler}
Let $(F,H)$ be general polarized hyperk\"ahler fourfold in $M_{2m}^{(1)}$, and consider $ X \in |H^{\otimes d}| $ smooth projective variety.

\begin{enumerate}
\item If $m=2$, then for every $d>4$,
$$ \Aut(X) \simeq \Aut(F; X) \simeq \Aut(\mathbb{P}^9; X).$$
\item If $m=3$, then for every $d>3$,
 Then, $$ \Aut(X) \simeq \Aut(F; X) \simeq \Aut(\mathbb{P}^{14}; X).$$
\end{enumerate}
 
\end{cor}
\begin{proof}
   Consequence of \cite[Theorem E]{ORS} and
Corollary \ref{cor:main}.
\end{proof}
\subsection{Other examples}

\subsubsection{Hypersurfaces in \texorpdfstring{$\mathbb{P}^N$}{P^N}}

\begin{thm}
Let \(F \subset \mathbb{P}^N\) be a smooth hypersurface of degree \(D\), $\dim  F \geq 3$ such that $D\leq N+1$ and let \(X \subset F\) be a hypersurface of degree \(d > D\) and $d \neq N+1-D$. Then \[\Aut(X) \simeq \Aut(F; X) \simeq \Aut(\mathbb{P}^N;X).
\]
\end{thm}

\begin{proof}
As a consequence of the Proposition \ref{coxring} and Proposition \ref{h1vanish} the section ring of \(\mathcal{O}_F(1)\) is isomorphic, as a graded ring, to \[
\mathbb{C}[T_0, \dots, T_N]/(f),
\]
where \(f\) is the defining equation of \(F\). The result follows by applying Corollary \ref{cor:main} to the inclusion \(X \subset F\).
\end{proof}

\begin{rmk}
The finiteness of \(\Aut(X)\) for complete intersections is already known (see \cite{Autcompleteintersections}). Moreover, for a general \(X\), one has \(\Aut(X)=\{1\}\) (see \cite{Automorphismandcohomology}, Theorem 1.3 and Remark 1.4).
\end{rmk}

We now compute an upper bound to $t_1(R(F,H))$ in the case that $F$ is a particular type of Calabi--Yau, namely a Wehler variety.
Wehler varieties have appeared in the literature in several contexts. For example, Cantat and Oguiso \cite{OguisoCantat} showed that they provide the first family of Calabi--Yau varieties for which $\text{Bir}(X)$ is infinite and the Kawamata-Morrison movable cone conjecture is satisfied. More recently, Filip, Lesieutre and Tosatti \cite{volume}
use them to describe some unexpected behaviour of  the volume function of a divisor.

\subsubsection{Wehler Varieties}

\begin{df} A Wehler variety is a smooth hypersurface $F \subset (\mathbb{P}^1)^{n+1}$ of degree $(2, \dots,2)$.
\end{df}

\begin{prop} \label{Wehler} Consider $F$ to be Wehler of dimension $n$. Then, $t_1(R(F,\OO_F(1,\dots,1))) =2 $.  
\end{prop}
\begin{proof}
We denote $\mathbb{P}:=(\mathbb{P}^1)^{n+1}$, to simplify our notation. Let $s \in H^0(\mathbb{P}, \OO_{\mathbb{P}}(2, \dots,2))$.For $m >0$, we can apply Kodaira vanishing, \[
H^1\!\left(\mathbb{P},
\mathcal{O}(m-2,\dots,m-2)\right) = 0 \text{, since } K_{\mathbb{P}}=\mathcal{O}(-2,\dots,-2). \]
For $m=0$, $H^1\!(\mathbb{P}, K_{\mathbb{P}}) = 0$, since $\mathbb{P}$ is Fano, as a consequence of Kodaira vanishing and Serre's duality.
By Proposition \ref{coxring}, 
$$
\bigoplus_{m \ge 0}H^0\!\left(  F,\mathcal{O}_{F}(m,\dots,m)\right)
\simeq
\frac{ \bigoplus_{m \ge 0}
H^0\!\left(\mathbb{P},
\mathcal{O}_{\mathbb{P}}(m,\dots,m)\right)
}{
(s) 
}$$
Since $\OO_{\mathbb{P}}(1, \dots, 1)$ is a normally presented line bundle by results on Table \ref{N1table}
$$A:=\bigoplus_{m \geq 0} H^0(\mathbb{P},\,
\mathcal{O}_{\mathbb{P}}(m,\dots,m))\big)
\xrightarrow[\simeq]{\alpha} \frac{\mathbb{C}[z_0, \dots, z_{N}]}{I_{SV}},
$$

\noindent where $\deg z_i =1$ and $I_{SV}$ is generated by degree 2 elements and $\alpha$ is the graded isomorphism induced by the Segre embedding $\mathbb{P}^1\times \dots \times \mathbb{P}^1 \xhookrightarrow{|O(1,\dots,1)|}\mathbb{P}^{N}$.
As $J = (s)$ is an ideal of $A$ generated by  an element of degree $2$, and $\alpha$ is an isomorphism of graded rings, we have that the degree of the generators of $\alpha(J)$ is  $2$.
Denoting by $I \subset \mathbb{C}[z_0, \dots, z_{N}] $ the ideal such that $\alpha(J) = I/I_{SV}$, we conclude that 
\[
\bigoplus_{m \geq 0}
H^0\!\left(F,\mathcal{O}_{F}(m,\dots,m)\right) \simeq
\frac{\frac{\mathbb{C}[z_0, \dots, z_{N}]}{I_{SV}}}{ \alpha(J) }
\simeq
\frac{\mathbb{C}[z_0, \dots, z_{N}]}{I}
\]

\noindent where  the maximum degree of the generators of $I$ is 2. Since $F$ is embedded in $\mathbb{P}^N$ by complete linear system $|\OO_F(1,\dots,1)|$, its image in $\mathbb{P}^N$ is nondegenerate, hence $t_1(R(F, \OO_F(1,\dots,1)))>1$, concluding our assertion.
\end{proof}

\begin{cor} \label{Wehler2}
Let $F$ be a Wehler variety with $\dim F  \geq 3$. For
$d>2$, let $X \in |\OO_F(d,\dots,d)|$ a smooth member. Then: $$\Aut(X)\simeq \Aut(F;X) \simeq\Aut(\mathbb{P}^N;X),$$ 
 where $i \colon F \hookrightarrow \mathbb{P}^N$ is the embedding induced by the complete linear system $|\OO_F(1,\dots,1)|$.    
\end{cor}

\begin{proof} By Proposition \ref{restrictionNG}, $\OO_F(1,\dots,1)$ is normally generated as a consequence of the normal generation of $\OO_{\mathbb{P}}(1, \dots 1)$. Then, the result follows from Corollary \ref{cor:main} and Proposition \ref{Wehler}.
\end{proof}

\

In the last example of this paper, we compute an upper bound of $t_1(R(F,-K_F))$ 
for Fano hypersurfaces of type $X_{(a_1, \dots,a_k)} \subset \mathbb{P}^{n_1} \times \dots \times \mathbb{P}^{n_k}$. The motivation behind our choice is very simple: many higher-dimensional Fano varieties can be realized in this form.

\subsubsection{Fano hypersurfaces of type $X_{(a_1, \dots,a_k)} \subset \mathbb{P}^{n_1} \times \dots \times \mathbb{P}^{n_k}$}

\begin{prop} \label{hypersurfacesinproducts}
Let $D = \OO_{\mathbb{P}}(a_1,\dots,a_j)$ be a divisor on 
$\mathbb{P} := \mathbb{P}^{n_1}\times \dots \times \mathbb{P}^{n_j}$ with $a_i \ge 0$, at least one $a_i>0$. Suppose that $(\OO_{\mathbb{P}}(D)\otimes K_\mathbb{P})^{-1}$ is ample, and $M:= \left\lceil
\max_{1\leq i\leq j}
\frac{a_i}{n_i+1-a_i}
\right\rceil.$ Then,
$$
t_1(R(D,-K_D)) \leq \max\{2, M\}.$$
\end{prop}

\begin{proof}
Let $s \in H^0(\mathbb{P}, \OO_{\mathbb{P}}(D))$ define $D$. Then we have the exact sequence
\[
0 \to \OO_{\mathbb{P}}(-D) \xrightarrow{\cdot s} \OO_{\mathbb{P}} \to \OO_D \to 0.
\]
For $m\geq 0$, tensoring this sequence by $(\OO_{\mathbb{P}}(D)\otimes K_\mathbb{P})^{\otimes -m}$ gives
\[
0 \to \OO_{\mathbb{P}}(-(m+1)D)\otimes K_\mathbb{P}^{\otimes -m}
\xrightarrow{\cdot s}
(\OO_{\mathbb{P}}(D)\otimes K_\mathbb{P})^{\otimes -m}
\to K_D^{\otimes -m} \to 0.
\]

Since $(\OO_{\mathbb{P}}(D)\otimes K_\mathbb{P})^{-1}$ is ample, Kodaira vanishing implies
\[
H^1\left(
\mathbb{P},
(\OO_{\mathbb{P}}(D)^{-1}\otimes K_\mathbb{P}^{-1})^{\otimes (m+1)}
\otimes K_\mathbb{P}
\right)=0
\qquad \text{for }m\geq 0.
\]
Thus,
\[
\bigoplus_{m \geq 0} H^0(D, K_D^{-\otimes m})
\simeq 
\frac{
\bigoplus_{m\geq 0} H^0(\mathbb{P}, 
(\OO_{\mathbb{P}}(D)\otimes K_\mathbb{P})^{\otimes -m})
}{
\bigoplus_{m\geq 0} s \cdot H^0(\mathbb{P},
\OO_{\mathbb{P}}(-(m+1)D)\otimes K_\mathbb{P}^{\otimes -m})
}.
\]

Now,
\[
K_{\mathbb{P}} = \OO_{\mathbb{P}}(-n_1-1,\dots,-n_j-1),
\]
so
\[
\OO_{\mathbb{P}}(-(m+1)D)\otimes K_\mathbb{P}^{\otimes -m}
=
\OO_{\mathbb{P}}\bigl(
m(n_1+1)-(m+1)a_1,\dots,
m(n_j+1)-(m+1)a_j
\bigr).
\]

Therefore,
\[
H^0\left(
\OO_{\mathbb{P}}(-(m+1)D)\otimes K_\mathbb{P}^{\otimes -m}
\right)\neq 0
\]
if and only if
\begin{equation} \label{inequality}
m \geq \frac{a_i}{n_i+1-a_i}
\qquad \text{for all }i.
\end{equation}

Let $M$ be the smallest integer satisfying \eqref{inequality} for all $i$. Then
\[
s \cdot H^0\left(
\mathbb{P},
\OO_{\mathbb{P}}(-(m+1)D)\otimes K_\mathbb{P}^{\otimes -m}
\right)
=
\begin{cases} 
0, & m < M, \\[4pt]
\subset H^0\left(
\mathbb{P},
\OO_{\mathbb{P}}(-mD)\otimes K_\mathbb{P}^{\otimes -m}
\right), & m \geq M.
\end{cases}
\]

We conclude that 
$$J =\bigoplus_{m\ge 0} s \cdot H^0(\mathbb{P}, \OO_{\mathbb{P}}(-(m+1)D)\otimes K_\mathbb{P}^{\otimes -m}) \subset\bigoplus_{m\ge M}  H^0(\mathbb{P}, (\OO_{\mathbb{P}}(D)\otimes K_\mathbb{P})^{\otimes -m}) $$ and so it is an ideal of 
$A:= \bigoplus_{m\ge 0} H^0(\mathbb{P}, (\OO_{\mathbb{P}}(D)\otimes K_\mathbb{P})^{\otimes -m})$.

Moreover,
\[
L:=(\OO_{\mathbb{P}}(D)\otimes K_\mathbb{P})^{-1}
=
\OO_{\mathbb{P}}(-a_1+n_1+1,\dots,-a_j+n_j+1).
\]Since $(\OO_{\mathbb{P}}(D)\otimes K_\mathbb{P})^{-1}$ is ample, we have
\[
-a_i+n_i+1>0
\]
for every $i$.
We have then that the ideal $J$ is generated by degree $M$ elements.

Indeed, since this is a consequence of the surjective of the following maps for
every $m \geq M$

$$ s \cdot H^0(
\mathbb{P},  \OO_{\mathbb{P }}(-D) \otimes L^{\otimes M}
) \otimes H^0(
\mathbb{P},
L^{\otimes (m -M)}
) \to s \cdot H^0(
\mathbb{P},
\OO_{\mathbb{P}}(-D) \otimes L^{\otimes m}).$$

Also, by the second line of the table \ref{N1table}, 
$L$ is a normally presented line bundle. Arguing as in the previous example, we obtain
\[
\bigoplus_{m \ge 0} H^0(D, K_D^{-\otimes m})
\simeq \frac{\mathbb{C}[z_0, \dots, z_n]}{I},
\]
where $\deg z_i=1$ and $I$ is generated by elements of degree $M$ and 2.
\end{proof}

\begin{cor} Let $F$ be a smooth
Fano hypersurface of type $X_{(a_1, \dots,a_k)} \subset \mathbb{P}^{n_1} \times \dots \times \mathbb{P}^{n_k}$. For 
$d > \max\{2,\left\lceil
\max_{1\leq i\leq j}
\frac{a_i}{n_i+1-a_i}
\right\rceil\}
$, let $ X \in|-dK_F|$ a smooth member. Then, 
\[\Aut(X) \simeq \Aut(F; X) \simeq \Aut(\mathbb{P}^N; X).
\]

\end{cor}
\begin{proof} By Proposition \ref{restrictionNG},
$-K_F$ is normally generated, so the result follows from Corollary \ref{cor:main}  and Proposition \ref{hypersurfacesinproducts}.
\end{proof}

\address{Dipartimento di Matematica - Università degli Studi di Ferrara, Via Nicolò Machiavelli, 30, 44121 Ferrara FE, Italia. ORCID: 0000-0002-1279-4077.}

\bibliographystyle{plain}

\bibliography{ref}

\end{document}